\documentclass[12pt]{amsart}
\usepackage{a4wide,enumerate,color,graphicx}
\usepackage{amsmath}
\usepackage{bbm}
\usepackage{comment}
\allowdisplaybreaks

\let\pa\partial  
\let\na\nabla  
\let\eps\varepsilon  
\newcommand{\N}{{\mathbb N}}  
\newcommand{\R}{{\mathbb R}} 
 
\newcommand{\diver}{\operatorname{div}}

\newtheorem{theorem}{Theorem}   
\newtheorem{lemma}[theorem]{Lemma}

\begin{document}  

\title[Vanishing cross-diffusion limit]{Vanishing cross-diffusion limit in a
Keller--Segel system with additional cross-diffusion}

\author[A. J\"ungel]{Ansgar J\"ungel}
\address{Institute for Analysis and Scientific Computing, Vienna University of  
	Technology, Wiedner Hauptstra\ss e 8--10, 1040 Wien, Austria}
\email{juengel@tuwien.ac.at} 

\author[O. Leingang]{Oliver Leingang}
\address{Institute for Analysis and Scientific Computing, Vienna University of  
	Technology, Wiedner Hauptstra\ss e 8--10, 1040 Wien, Austria}
\email{oliver.leingang@tuwien.ac.at} 

\author[S. Wang]{Shu Wang}
\address{College of Applied Science, Beijing University of Technology, 
Beijing, PR China}
\email{wangshu@bjut.edu.cn} 

\date{\today}

\thanks{The first two authors acknowledge partial support from   
the Austrian Science Fund (FWF), grants P30000, W1245, and F65.
The third author is partially supported by the National Natural Science Foundation 
of China (NSFC), grants 11831003, 11771031, 11531010, and by NSF of Qinghai Province,
grant 2017-ZJ-908.} 

\begin{abstract}
Keller--Segel systems in two and three space dimensions with an additional 
cross-diffusion term in the equation for the chemical concentration are analyzed.
The cross-diffusion term has a stabilizing effect and leads to the global-in-time
existence of weak solutions.
The limit of vanishing cross-diffusion parameter is proved rigorously in the
parabolic-elliptic and parabolic-parabolic cases. When the signal production
is sublinear, the existence of global-in-time weak solutions as well as the
convergence of the solutions to those of the classical parabolic-elliptic
Keller--Segel equations are proved. The proof is based on a reformulation of
the equations eliminating the additional cross-diffusion term but making the
equation for the cell density quasilinear. For superlinear signal production terms,
convergence rates in the cross-diffusion parameter are proved for local-in-time
smooth solutions (since finite-time blow up is possible). The proof is based
on careful $H^s(\Omega)$ estimates and a variant of the Gronwall lemma.
Numerical experiments in two space dimensions
illustrate the theoretical results and quantify the shape of the
cell aggregation bumps as a function of the cross-diffusion parameter.
\end{abstract}

 

\maketitle


\section{Introduction}

Chemotaxis describes the directed movement of cells in response to chemical gradients
and may be modeled by the (Patlak--) Keller--Segel equations \cite{KeSe70,Pat53}. 
The aggregation
of cells induced by the chemical concentration is counter-balanced by cell diffusion. 
If the cell density is sufficiently large, the chemical interaction dominates diffusion
and results in a blow-up of the cell density. However, a single point blow-up is not very realistic from a biological view point,
due to the finite size of the cells. Therefore, chemotaxis models
that avoid blow-up were suggested in the literature. 
Possible approaches are bounded chemotaxis sensibilities
to avoid over-crowding \cite{BDD06,DiRo08}, degenerate cell diffusion 
\cite{CaCa06,Kow05,KoSz08}, death-growth terms \cite{BHLM96,Win08},
or additional cross-diffusion \cite{CHJ12,HiJu11}. In this paper, we continue
our study of the Keller--Segel system with additional cross-diffusion, which allows
for global weak solutions \cite{CHJ12,HiJu11}.
The question how well the solutions approximate those from the
original Keller--Segel system remained open. Here, we will answer this
question by performing rigorously the vanishing cross-diffusion limit and
giving an estimate for the difference of the respective solutions.

The equations for the cell density $\rho_\delta(x,t)$ and the chemical concentration
$c_\delta(x,t)$ are given by
\begin{equation}\label{1.eq}
  \pa_t\rho_\delta = \diver(\na\rho_\delta-\rho_\delta\na c_\delta), \quad 
	\eps\pa_t c_\delta = \Delta c_\delta + \delta\Delta\rho_\delta 
	- c_\delta + \rho_\delta^\alpha
	\quad\mbox{in }\Omega,\ t>0,
\end{equation}
subject to the no-flux and initial conditions
\begin{equation}\label{1.bic}
  \na\rho_\delta\cdot\nu = \na c_\delta\cdot\nu = 0\quad\mbox{on }\pa\Omega,\ t>0, \quad
	\rho_\delta(0)=\rho^0,\ \eps c_\delta(0)=\eps c^0\quad\mbox{in }\Omega,
\end{equation}
where $\Omega\subset\R^d$ ($d=2,3$) is a bounded domain,
$\nu$ is the exterior unit normal vector of $\pa\Omega$,
$\delta>0$ describes the strength of the additional cross-diffusion, and 
the term $\rho_\delta^\alpha$ with $\alpha>0$ is the nonlinear signal production.
The case $\eps=1$ is called the parabolic-parabolic case, while $\eps=0$
refers to the parabolic-elliptic case. 
The existence of global weak solutions to \eqref{1.eq}--\eqref{1.bic}
was proved in \cite{HiJu11} in two space dimensions and in \cite{CHJ12} 
in three space dimensions (with degenerate diffusion). 

Setting $\delta=0$, we obtain the Keller--Segel system with nonlinear signal
production. In the classical Keller--Segel model, the signal production is assumed 
to be linear, $\alpha=1$. In order to deal
with the three-dimensional case, we need sublinear signal productions, 
$\alpha<1$.
In fact, it is shown in \cite{Win18} that $\alpha=2/d$ is the critical
value for global existence versus finite-time blow-up 
in a slightly modified Keller--Segel system with 
$\delta=0$. We show that also for $\delta>0$, the condition $\alpha<2/d$
guarantees the global existence of weak solutions, while numerical results
indicate finite-time blow-up if $\alpha>2/d$.

We are interested in the limit $\delta\to 0$ in \eqref{1.bic} leading to the 
Keller--Segel equations
\begin{equation}\label{1.ks}
  \pa_t\rho = \diver(\na\rho-\rho\na c), \quad 
	\eps\pa_t c = \Delta c - c + \rho^\alpha\quad\mbox{in }\Omega,\ t>0,
\end{equation}
with the initial and boundary conditions \eqref{1.bic}. 
We refer to the reviews \cite{BBTW15,HiPa09} for references
concerning local and global solvability and results for variants of this
model. In the following, we recall only some effects related to the blow-up behavior.

In the parabolic-elliptic case, a dichotomy arises for \eqref{1.ks} 
in two space dimensions: If the initial mass $M:=\|\rho^0\|_{L^1(\Omega)}$ is smaller 
than $8\pi$, the solutions are global in time, while they blow up in finite time
if $M>8\pi$ and the second moment of the initial datum is finite (see, e.g.,
\cite{Bla12}). The condition on the second moment implies that the initial
density is highly concentrated around some point. It is necessary in the sense that
there exists a set of initial data with total mass larger than $8\pi$ such that
the corresponding solutions are global \cite{BaCa15}.

In the parabolic-parabolic case, again in two space dimensions
and with finite second moment, the solutions exist globally in time of 
$M<8\pi$ \cite{CaCa08}. However, in contrast to the parabolic-elliptic case, the
threshold value for $M$ is less precise, and solutions with large mass can exist
globally. In dimensions $d\ge 3$, a related critical phenomenon occurs: The solutions
exist globally in time if $\|\rho^0\|_{L^{d/2}(\Omega)}$ is sufficiently small,
but they blow up in finite time if the total mass is large compared to the
second moment \cite{CPZ04}.

In this paper, we prove two results. The first one is the convergence of the solutions
to the parabolic-elliptic model \eqref{1.eq} with $\eps=0$ to a solution
to the parabolic-elliptic Keller--Segel system \eqref{1.ks} with $\eps=0$.
Since we impose a restriction on the parameter $\alpha$ as mentioned above,
this result holds globally in time. We set $Q_T=\Omega\times(0,T)$.

\begin{theorem}[Convergence for the parabolic-elliptic model]\label{thm.pe}
Let $\Omega\subset\R^d$ $(d=2,3)$ be bounded with $\pa\Omega\in C^{1,1}$,
$T>0$, $\delta>0$, $\eps=0$, and $0\le\rho^0\in L^\infty(\Omega)$. 
If $\alpha<2/d$, there exists
a weak solution $(\rho_\delta,c_\delta)\in L^2(0,T;H^1(\Omega))^2$ to 
\eqref{1.eq}-\eqref{1.bic} satisfying
$$
  \pa_t\rho_\delta\in L^2(0,T;H^1(\Omega)'), \quad 
	\rho_\delta\in L^{\infty}(0,T;L^3(\Omega)), \quad
	c_\delta+\delta\rho_\delta\in L^\infty(0,T;W^{1,p}(\Omega)),
$$
for any $p<\infty$. Furthermore, as $\delta\to 0$,
\begin{align*}
  \rho_\delta\to\rho &\quad\mbox{strongly in }L^2(Q_T), \\
	\na\rho_\delta\rightharpoonup\na\rho &\quad\mbox{weakly in }L^2(Q_T), \\
	c_\delta+\delta\rho_\delta\rightharpoonup^* c &\quad\mbox{weakly* in }
	L^\infty(0,T;W^{1,p}(\Omega)), \ p<\infty,
\end{align*}
and $(\rho,c)$ solves \eqref{1.bic}--\eqref{1.ks}, and it holds that
$\rho$, $c\in L^\infty(0,T;L^{\infty}(\Omega))$.
\end{theorem}

The idea of the proof is to reformulate \eqref{1.eq} via 
$v_\delta:=c_\delta+\delta\rho_\delta$ as the system
\begin{equation}\label{1.rhov}
  \pa_t\rho_\delta = \diver((1+\delta\rho_\delta)\na\rho_\delta
	-\rho_\delta\na v_\delta), \quad
	-\Delta v_\delta+v_\delta = \delta \rho_\delta + \rho_\delta^\alpha
	\quad\mbox{in }\Omega,\ t>0,
\end{equation}
together with the initial and boundary conditions
\begin{equation}\label{1.bic2}
  \na\rho_\delta\cdot\nu = \na v_\delta\cdot\nu = 0\quad\mbox{in }\pa\Omega, \quad
	\rho_\delta(0)=\rho^0\quad\mbox{in }\Omega.
\end{equation}
This reformulation was already used in \cite{HiJu11} to prove the existence of
weak solutions in the two-dimension case with $\alpha=1$. 
It transforms the asymptotically singular problem
to a quasilinear parabolic equation, thus simplifying
considerably the asymptotic limit problem. Still, we need estimates uniform in
$\delta$ to apply compactness arguments. For this, we use the ``entropy''
functional $H_1(\rho_\delta)=\int_\Omega\rho_\delta(\log\rho_\delta-1)dx$:
\begin{equation}\label{1.H1}
  \frac{dH_1}{dt} + 4\int_\Omega|\na\rho_\delta^{1/2}|^2 dx
	+ \delta\int_\Omega|\na\rho_\delta|^2 dx 
	\le \int_\Omega(\delta\rho_\delta^2 + \rho_\delta^{\alpha+1})dx.
\end{equation}
By the Gagliardo--Nirenberg inequality, the first term on the right-hand side
can be estimated as (see Section \ref{sec.pe} for the proof)
$$
  \delta\int_\Omega\rho_\delta^2 dx 
	\le \frac{\delta}{2}\|\na\rho_\delta\|_{L^2(\Omega)}^2
	+ C\|\rho_\delta\|_{L^1(\Omega)}^2.
$$
The first term on the right-hand side is absorbed by the left-hand side of
\eqref{1.H1} and the second term is bounded since the total mass is conserved.
To estimate the second term on the right-hand side of \eqref{1.H1}, we need
another ``entropy'' functional $H_p(\rho)=\int_\Omega\rho_\delta^p dx$ with
$p=2$ or $p=3$, leading to
\begin{equation}\label{1.Hp}
  \frac{dH_p}{dt} + \int_\Omega|\na\rho_\delta^{p/2}|^2 dx
	+ \delta\int_\Omega|\na\rho_\delta^{(p+1)/2}|^2 dx
	\le C\|\rho_\delta\|_{L^{(p+1)/2}(\Omega)}^{(p+1)/2}
	+ C\|\rho_\delta^{p/2}\|_{L^1(\Omega)}^{\beta(p)},
\end{equation}
where $\beta(p)\ge 2$ is some function depending on $p$. If $p=2$, the last term is the 
total mass which is bounded uniformly in time. Moreover, the estimate for 
$\rho_\delta^{1/2}$ in $H^1(\Omega)$ from \eqref{1.H1}
implies that $\rho_\delta$ is bounded
in $W^{1,1}(\Omega)\hookrightarrow L^{3/2}(\Omega)$ such that the first term on the
right-hand side of \eqref{1.Hp} is uniformly bounded as well. Higher-order
integrability is then obtained for $p=3$. 

Clearly, these arguments are formal. In particular, the estimate for $H_1$
requires the test function $\log\rho_\delta$ in \eqref{1.eq} which may be
not defined if $\rho_\delta=0$. 
Therefore, we consider an implicit Euler discretization in time
with parameter $\tau>0$ and an elliptic regularization in space with parameter
$\eta>0$ to prove first the existence of approximate
weak solutions with strictly positive $\rho_\delta$. This is done by using
the entropy method of \cite{HiJu11}. The approximate solutions also satisfy
the $\delta$-uniform bounds, and they hold true when passing to the
limit $(\eta,\tau)\to 0$. Then the limit $\delta\to 0$ can be performed by
applying the Aubin--Lions lemma and weak compactness arguments.

The second result is concerned with the derivation of a convergence rate
both in the parabolic-parabolic and parabolic-elliptic case.
For $\alpha\ge 1$, we cannot generally expect global solutions. In this case, it
is natural to consider local solutions. The technique requires smooth solutions
so that we need some regularity assumptions on the initial data.

\begin{theorem}[Convergence rates]\label{thm.pp}
Let $\Omega\subset\R^d$ $(d\le 3)$ be a bounded domain with smooth boundary
and let $(\rho^0,c^0)\in (W^{1,p}(\Omega))^2$ for $p>d$ if $\eps=1$ and
$\rho^0\in C^{2+\gamma}(\overline\Omega)$ for some $\gamma\in(0,1)$ if $\eps=0$. 
Furthermore,
let $\alpha=1$ or $\alpha\ge 2$ and let $(\rho_\delta,c_\delta)$ and $(\rho,c)$
be (weak) solutions to \eqref{1.eq} and \eqref{1.ks}, \eqref{1.bic}, respectively,
with the same initial data.
Then these solutions are smooth locally in time and
there exist constants $C>0$ and $\delta_0>0$
such that for all $0<\delta<\delta_0$ and $\lambda>0$,
\begin{equation}\label{1.diff}
  \|(\rho_\delta-\rho,c_\delta-c)\|_{L^\infty(0,T;H^2(\Omega))}
	\le C\delta^{1-\lambda}.
\end{equation}
\end{theorem}

The theorem is proved as in, e.g., \cite{HsWa06} by deriving carefully 
$H^s(\Omega)$ estimates for the difference
$(\rho_R,c_R):=(\rho_\delta-\rho,c_\delta-c)$.
The index $s\in\N$ is chosen such that we obtain $L^\infty(\Omega)$
estimates which are needed to handle the nonlinearities. 
If $\eps=1$, we introduce the functions
$$
  \Gamma(t) = \|(\rho_R,c_R)(t)\|_{H^2(\Omega)}^2, \quad
	G(t) = \|(\rho_R,c_R)(t)\|_{H^2(\Omega)}^2 
	+ \|\na\Delta(\rho_R,c_R)(t)\|_{L^2(\Omega)}^2.
$$
The aim is to prove the inequality
$$
  \Gamma(t) + C_1\int_0^t G(s)ds 
	\le C_2\int_0^t(\Gamma(s)+\Gamma(s)^{\max\{2,\alpha\}})ds
	+ C_2\int_0^t \Gamma(s)G(s)ds + C_2\delta^2,
$$
where $C_1$, $C_2>0$ are constants independent of $\delta$. 
This inequality allows us to apply
a variant of Gronwall's lemma (see Lemma \ref{lem.gronwall} in the Appendix), 
implying \eqref{1.diff}. 
In the parabolic-elliptic case $\eps=0$, the functions $\Gamma(t)$ and $G(t)$
are defined without $c_R$, and the final inequality contains the
additional integral $\int_0^t\Gamma(s)^{\alpha+1}G(s)ds$, which is still
covered by Lemma \ref{lem.gronwall}. The condition $\alpha=1$ or $\alpha\ge 2$ comes
from the fact that the derivative of the mapping $s\mapsto s^\alpha$ is 
H\"older continuous
exactly for these values. The numerical results in Section \ref{sec.num} indicate
that the convergence result may still hold for $\alpha\in(1,2)$. 

The optimal convergence rate is expected to be one. Theorem \ref{thm.pp} provides 
an almost optimal rate. The reason for the non-optimality comes from the variant
of the nonlinear Gronwall lemma proved in Lemma \ref{lem.gronwall}. We conjecture that
an optimal rate holds (changing the constants in Lemma \ref{lem.gronwall}), but
since this issue is of less interest, we did not explore it further.

The theoretical results are illustrated by numerical experiments, using the
software tool NGSolve/Netgen. We remark that numerical tests for model
\eqref{1.eq} were already performed in \cite{BeJu13,HiJu11}.
For positive values of $\delta$, the (globally existing)
cell density forms bumps at places
where the solution to the classical Keller--Segel system develops an 
$L^\infty(\Omega)$ blow up.
Compared to \cite{BeJu13,HiJu11}, we investigate the dependence of the shape
of the bumps on $\delta$. In a radially symmetric situation,
it turns out that the radius of the bump (more precisely the diameter of
a level set $\rho_\delta\approx 0$) behaves like $\delta^a$ with $a\approx 0.43$,
and the maximum of the bump behaves like $\delta^{-b}$ with 
$b\approx 1.00$. 

The paper is organized as follows. Theorem \ref{thm.pe} is proved in Section
\ref{sec.pe}, while Theorem \ref{thm.pp} is shown in Section \ref{sec.pp}.
The numerical experiments for model \eqref{1.eq} are performed in
Section \ref{sec.num}. Some technical tools, including the nonlinear
Gronwall inequality, are recalled in the Appendix.


\section{Proof of Theorem \ref{thm.pe}}\label{sec.pe}

We prove Theorem \ref{thm.pe} by an approximation procedure and by deriving the
uniform bounds from discrete versions of the entropy inequalities
\eqref{1.H1} and \eqref{1.Hp}.

{\em Step 1: Solution of a regularized system and entropy estimates.}
We show the existence of solutions to a discretized and regularized
version of \eqref{1.rhov}. For this, let $N\in\N$ and $\tau=T/N$, and set
$\rho(w)=\exp(w/\delta)$. This means that we transform $w=\delta\log\rho$.
Let $w^{k-1}\in H^2(\Omega)$ and $v^{k-1}\in H^1(\Omega)$ be given
and set $\rho^{j}=\rho(w^{j})$ for $j=k,k-1$. Consider for given $\tau>0$
and $\eta>0$ the regularized system
\begin{align}
  &\frac{1}{\tau}\int_\Omega(\rho^k-\rho^{k-1})\phi dx
  + \int_\Omega\big((1+\delta\rho^k)\na\rho^k - \rho^k\na v^k\big)\cdot\na\phi dx 
	\nonumber \\
  &\phantom{xxxx}{}= -\eta\int_\Omega\big(\Delta w^k \Delta\phi
	+ \delta^{-2}|\na w^k|^2\na w^k\cdot\na \phi + w^k\rho^k \phi\big)dx, \label{2.wrho} \\
	&\int_\Omega(\na v^k\cdot\na\theta + v^k\theta)dx 
	= \int_\Omega\big(\delta\rho^k + (\rho^k)^\alpha\big)\theta dx \label{2.wc}
\end{align}
for $\phi\in H^2(\Omega)$ and $\theta\in H^1(\Omega)$. 
The time discretization is needed to handle issues due to low time regularity,
while the elliptic regularization guarantees $H^2(\Omega)$ solutions which, by
Sobolev embedding (recall that $d\le 3$), are bounded. The higher-order gradient
term $|\na w^k|^2\na w^k\cdot\na\phi$ is necessary to derive 
$L^p(\Omega)$ estimates. The existence of a solution
$w^k\in H^2(\Omega)$, $0\le v^k\in H^1(\Omega)$ follows from the techniques
used in the proof of Proposition 3.1 in \cite{HiJu11} 
employing the Leray--Schauder fixed-point theorem. Since these techniques are
rather standard now, we omit the proof and refer to \cite{HiJu11,Jue15,Jue16}
for similar arguments.

Inequality $w^k\rho^k=w^k e^{w^k/\delta}\ge e^{w^k/\delta}-1=\rho^k-1$ allows
us to show as in \cite[page 1004]{HiJu11} that the total mass 
$\|\rho^k\|_{L^1(\Omega)}$ is bounded uniformly in $\delta$.

Entropy estimates are derived from \eqref{2.wrho}--\eqref{2.wc} 
by choosing the test functions $\phi=w^k/\delta=\log\rho^k$ and $\theta=\rho^k$, 
respectively, and adding both equations.
Then the terms involving $\na v^k$ cancel and after some elementary computations,
we end up with
\begin{align}
  \frac{1}{\tau}\int_\Omega & \big(h(\rho^k)-h(\rho^{k-1})\big)dx
	+ 4\int_\Omega|\na(\rho^k)^{1/2}|^2 dx + \delta\int_\Omega|\na\rho^k|^2 dx 
	\nonumber \\
	&\phantom{xx}{}+ \frac{\eta}{\delta}\int_\Omega\big((\Delta w^k)^2
	+ \delta^{-2}|\na w^k|^4 + (w^k)^2\rho^k\big)dx \nonumber \\
	&\le \int_\Omega\big(-v^k\rho^k + \delta(\rho^k)^2 + (\rho^k)^{\alpha+1}\big) dx
	\le \int_\Omega\big(\delta(\rho^k)^2 + (\rho^k)^{\alpha+1}\big) dx, \label{2.ineq}
\end{align}
where $h(s)=s(\log s-1)$ for $s\ge 0$. Using the Gagliardo--Nirenberg inequality
with $\sigma=d/(d+2)$, the Poincar\'e--Wirtinger inequality, and then the
Young inequality with $p=1/\sigma$, $p'=1/(1-\sigma)$, it holds for any
$u\in H^1(\Omega)$ and $\mu>0$ that
\begin{align}
  \|u\|_{L^2(\Omega)}^2
	&\le C\|u\|_{H^1(\Omega)}^{2\sigma}
	\|u\|_{L^1(\Omega)}^{2(1-\sigma)}
	= C\big(\|\na u\|_{L^2(\Omega)}^2 + \|u\|_{L^1(\Omega)}^2\big)^{\sigma}
	\|u\|_{L^1(\Omega)}^{2(1-\sigma)} \nonumber \\
	&\le \mu\big(\|\na u\|_{L^2(\Omega)}^2 + \|u\|_{L^1(\Omega)}^2\big) 
	+ C(\mu)\|u\|_{L^1(\Omega)}^2
	= \mu\|\na u\|_{L^2(\Omega)}^2 + C(\mu)\|u\|_{L^1(\Omega)}^2. \label{2.u}
\end{align}
We deduce from this inequality that the first term on the right-hand side 
of \eqref{2.ineq} can be estimated as
$$
  \delta\int_\Omega(\rho^k)^2 dx
	\le \frac{\delta}{4}\|\na\rho^k\|_{L^2(\Omega)}^2 + C\|\rho^k\|_{L^1(\Omega)}^2,
$$
where here and in the following, $C>0$ denotes a constant, independent of 
$\eta$, $\tau$, and $\delta$, with values varying from line to line.
The first term on the right-hand side can be absorbed by the left-hand side 
of \eqref{2.ineq}, while the second term is bounded.
Using $s^{\alpha+1}\le s^2 +1$ for $s\ge 0$ (since $\alpha<1$) and \eqref{2.u},
the second term on the right-hand side of \eqref{2.ineq} becomes
$$
  \int_\Omega(\rho^k)^{\alpha+1}dx \le \int_\Omega\big((\rho^k)^2+1\big) dx 
	\le \frac{\delta}{4}\|\na\rho^k\|_{L^2(\Omega)}^2 
	+ C\|\rho^k\|_{L^1(\Omega)}^2 + C(\Omega).
$$
Inserting these estimations into \eqref{2.ineq}, we conclude that
\begin{align}
  \frac{1}{\tau}\int_\Omega & \big(h(\rho^k)-h(\rho^{k-1})\big)dx
	+ 4\int_\Omega|\na(\rho^k)^{1/2}|^2 dx 
	+ \frac{\delta}{2}\int_\Omega|\na\rho^k|^2 dx 
	\nonumber \\
	&\phantom{xx}{}+ \frac{\eta}{\delta}\int_\Omega\big((\Delta w^k)^2
	+ \delta^{-2}|\na w^k|^4 + (w^k)^2\rho^k\big)dx
	\le C(\Omega). \label{2.ei}
\end{align}

{\em Step 2: Further uniform estimates.}
The estimates from \eqref{2.ei} are not sufficient for the limit $\delta\to 0$,
therefore, we derive further uniform bounds.
Let $p=2$ or $p=3$. We choose the admissible test functions 
$p(\rho^k)^{p-1}$ and $(p-1)(\rho^k)^p$
in \eqref{2.wrho}--\eqref{2.wc}, respectively, 
and add both equations. The convexity of $s\mapsto s^p$ implies that
$s^p-t^p\le p(s-t)s^{p-1}$. Then, observing that the terms involving $\na v^k$
cancel, we find that
\begin{align}
  \frac{1}{\tau}\int_\Omega&\big((\rho^k)^p-(\rho^{k-1})^p\big) dx
	+ \frac{4}{p}(p-1)\int_\Omega|\na(\rho^k)^{p/2}|^2 dx
	+ \delta\frac{4p(p-1)}{(p+1)^2}\int_\Omega|\na(\rho^k)^{(p+1)/2}|^2dx \nonumber \\
	&\phantom{xx}{}
	+ \eta p\int_\Omega\big(\Delta w^k\Delta (\rho^k)^{p-1}
	+ \delta^{-2}|\na w^k|^2\na w^k\cdot\na(\rho^k)^{p-1} + w^k(\rho^k)^{p}\big)dx 
	\nonumber \\
	&\le (p-1)\int_\Omega\big(\delta(\rho^k)^{p+1} + (\rho^k)^{p+\alpha}\big)dx.
	\label{2.aux}
\end{align}
By \eqref{2.u}, we find that
$$
  \delta\int_\Omega(\rho^k)^{p+1}dx = \delta\|(\rho^k)^{(p+1)/2}\|_{L^2(\Omega)}^2
	\le \frac{\delta}{2}\|\na(\rho^k)^{(p+1)/2}\|_{L^2(\Omega)}^2
	+ C\|(\rho^k)^{(p+1)/2}\|_{L^1(\Omega)}^2.
$$
The estimate for $(\rho^k)^{\alpha+p}$ requires that $\alpha<2/d$. Indeed,
we deduce from the Gagliardo--Nirenberg inequality with 
$\sigma=d(2\alpha+p)/((d+2)(\alpha+p))$ and similar arguments as in \eqref{2.u} that
\begin{align*}
  \int_\Omega(\rho^k)^{\alpha+p}dx 
	&= \|(\rho^k)^{p/2}\|_{L^{2(\alpha+p)/p}(\Omega)}^{2(\alpha+p)/p}
	\le C\|(\rho^k)^{p/2}\|_{H^1(\Omega)}^{2(\alpha+p)\sigma/p}
  \|(\rho^k)^{p/2}\|_{L^1(\Omega)}^{2(\alpha+p)(1-\sigma)/p} \\
	&= C\big(\|\na(\rho^k)^{p/2}\|_{L^2(\Omega)}^2 + \|(\rho^k)^{p/2}\|_{L^1(\Omega)}^2
	\big)^{(\alpha+p)\sigma/p}\|(\rho^k)^{p/2}\|_{L^1(\Omega)}^{2(\alpha+p)(1-\sigma)/p} \\
	&\le \frac12\big(\|\na(\rho^k)^{p/2}\|_{L^2(\Omega)}^2 
	+ \|(\rho^k)^{p/2}\|_{L^1(\Omega)}^2
	\big) + C\|(\rho^k)^{p/2}\|_{L^1(\Omega)}^{\beta(p)},
\end{align*}
where $\beta(p):=2(\alpha+p)(1-\sigma)/((\alpha+p)(1-\sigma)-\alpha)\ge 2$
(the exact value of $\beta(p)$ is not important in the following).
For the last step, we used
the crucial inequality $(\alpha+p)\sigma/p<1$ (which is equivalent to
$\alpha<p/d$).

Because of $\rho^k=\exp(w^k/\delta)$, a computation shows that the integral 
with factor $\eta$ can be written as
\begin{align*}
  \eta &p\int_\Omega\big(\Delta w^k\Delta (\rho^k)^{p-1}
	+ \delta^{-2}|\na w^k|^2\na w^k\cdot\na(\rho^k)^{p-1} + w^k(\rho^k)^{p}\big)dx \\
	&= \frac{\eta}{\delta} p(p-1)\int_\Omega(\rho^k)^{p-1}\bigg[
	\bigg(\Delta w^k + \frac{p-1}{2\delta}|\na w^k|^2\bigg)^2
	+ \frac{1}{4\delta^2}\big(4-(p-1)^2\big)|\na w^k|^4\bigg]dx \\
	&\phantom{xx}{}+ \eta p\int_\Omega w^k e^{pw^k/\delta}dx.
\end{align*}
The last integral is bounded from below, independently of $\delta$. 
Since $p=2$ or $p=3$, 
the first integral on the right-hand side is nonnegative. (At this point, we need
the term $|\na w^k|^2\na w^k\cdot\na\phi$.) Summarizing these estimates, we infer that
\begin{align}
  \frac{1}{\tau}\int_\Omega&\big((\rho^k)^p-(\rho^{k-1})^p\big) dx
	+ \int_\Omega|\na(\rho^k)^{p/2}|^2 dx
	+ \frac{\delta}{4}\int_\Omega|\na(\rho^k)^{(p+1)/2}|^2dx \nonumber \\
  &\le C\|(\rho^k)^{(p+1)/2}\|_{L^1(\Omega)}^2 
	+ C\|(\rho^k)^{p/2}\|_{L^1(\Omega)}^{\beta(p)}. \label{2.est2}
\end{align}

{\em Step 3: Limit $(\eta,\tau)\to 0$.} 
Let $w_\tau(x,t)=w^k(x)$, $\rho_\tau(x,t)=\rho(w^k(x))$, $v_\tau(x,t)=v^k(x)$ 
for $x\in\Omega$ and $t\in((k-1)\tau,k\tau]$, $k=1,\ldots,N$, 
be piecewise constant functions in time. 
At time $t=0$, we set $w_\tau(x,0)=\log\rho^0(x)$ and $\rho_\tau(x,0)=\rho^0(x)$
for $x\in\Omega$. (Here, we need $\rho^0\ge C>0$ in $\Omega$ and another approximation
procedure which we omit; see, e.g., \cite[proof of Theorem 4.1]{Jue16}.)
Furthermore, we introduce the shift operator
$\pi_\tau\rho_\tau(x,t)=\rho_\tau(x,t-\tau)$ for $x\in\Omega$, $t\ge\tau$.
Then the weak formulation \eqref{2.wrho}--\eqref{2.wc} can be written as
\begin{align}
  &\frac{1}{\tau}\int_0^T\int_\Omega(\rho_\tau-\pi_\tau\rho_\tau)\phi dxdt
	+ \int_0^T\int_\Omega\big((1+\delta\rho_\tau)\na\rho_\tau-\rho_\tau\na v_\tau\big)
	\cdot\na\phi dxdt \nonumber \\
	&\phantom{xxx}= -\eta\int_0^T\int_\Omega\big(\Delta w_\tau\Delta\phi
	+ \delta^{-2}|\na w_\tau|^2\na w_\tau\cdot\na\phi + w_\tau\rho_\tau\phi\big)dxdt, 
	\label{2.rhotau} \\
	&\int_0^T\int_\Omega(\na v_\tau\cdot\na\theta + v_\tau\theta)dxdt
	= \int_0^T\int_\Omega(\delta\rho_\tau + \rho_\tau^\alpha)\theta dxdt, 
	\label{2.vtau}
\end{align}
where $\phi:(0,T)\to H^2(\Omega)$ and $\theta:(0,T)\to H^1(\Omega)$
are piecewise constant functions.

Multiplying \eqref{2.ei} by $\tau$, summing over $k=1,\ldots,N$, and applying the
discrete Gronwall inequality \cite[Lemma A.2]{Jue16} provides the following uniform
estimates:
\begin{align}
  \|\rho_\tau\|_{L^\infty(0,T;L^1(\Omega))} 
	+ \|\rho_\tau^{1/2}\|_{L^2(0,T;H^1(\Omega))} 
	+ \delta^{1/2}\|\rho_\tau\|_{L^2(0,T;H^1(\Omega))} &\le C, \label{3.rhotau} \\
  \eta^{1/2}\|\Delta w_\tau\|_{L^2(Q_T)}
	+ \eta^{1/4}\delta^{-1/2}\|\na w_\tau\|_{L^4(Q_T)}
	+ \eta^{1/2}\|w_\tau\rho_\tau^{1/2}\|_{L^2(Q_T)} &\le C. \label{3.wtau}
\end{align}

The (simultaneous) limit $(\eta,\tau)\to 0$ does not require estimates 
uniform in $\delta$.
Therefore, we can exploit the bound for $\rho_\tau$ in $L^2(0,T;H^1(\Omega))$.
Together with the uniform $L^\infty(0,T;L^1(\Omega))$ bound, we obtain from
the Gagliardo--Nirenberg inequality as in \cite[page 95]{Jue16} that
$(\rho_\tau)$ is bounded in $L^{2+2/d}(Q_T)$, recalling that $Q_T=\Omega\times(0,T)$.
Since $-\Delta v_\tau + v_\tau = \delta\rho_\tau + \rho_\tau^\alpha$ is
bounded in $L^{2+2/d}(Q_T)$, we deduce from elliptic regularity a uniform bound for 
$v_\tau$ in $L^{2+2/d}(0,T;W^{2,2+2/d}(\Omega))$. Therefore,
$\rho_\tau\na v_\tau$ is uniformly bounded in $L^{1+1/d}(Q_T)$ and
$\rho_\tau\na\rho_\tau$ is uniformly bounded in $L^{(2d+2)/(2d+1)}(Q_T)$.
Consequently, $(\rho_\tau-\pi_\tau\rho_\tau)/\tau=\diver((1+\delta\rho_\tau)\na\rho_\tau
- \rho_\tau\na v_\tau)$ is uniformly bounded in $L^{(2d+2)/(2d+1)}
(0,T;W^{-1,(2d+2)/(2d+1)}(\Omega))$.

The Aubin--Lions lemma in the version of \cite{DrJu12} shows that there exists
a subsequence, which is not relabeled, such that, as $(\eta,\tau)\to 0$,
$$
  \rho_\tau\to\rho\quad\mbox{strongly in }L^2(0,T;L^p(\Omega))
$$
for any $p<6$ and in $L^q(Q_T)$ for any $q<2+2/d$. 
Moreover, because of the bounds \eqref{3.rhotau}, 
again for a subsequence, as $(\eta,\tau)\to 0$,
\begin{align*}
  \na\rho_\tau\rightharpoonup \na\rho &\quad\mbox{weakly in }L^2(Q_T), \\
	\tau^{-1}(\rho_\tau-\pi_\tau\rho_\tau) \rightharpoonup \pa_t\rho
	&\quad\mbox{weakly in }L^{(2d+2)/(2d+1)}(0,T;W^{-1,(2d+2)/(2d+1)}(\Omega)), \\
  v_\tau\rightharpoonup v &\quad\mbox{weakly in }L^{2+2/d}(0,T;W^{2,2+2/d}(\Omega)).
\end{align*}
We deduce that $\rho_\tau\na\rho_\tau\rightharpoonup\rho\na\rho$ and
$\rho_\tau\na v_\tau\rightharpoonup\rho\na v$ weakly in $L^1(Q_T)$ as well as
$\rho_\tau^\alpha\to\rho^\alpha$ strongly in $L^2(Q_T)$. 

The limit in the term involving $\eta$ is performed as in \cite{HiJu11}:
Estimates \eqref{3.wtau} imply that, for any $\phi\in L^4(0,T;H^2(\Omega))$,
\begin{align*}
  \bigg|\eta&\int_0^T\int_\Omega\big(\Delta w_\tau\Delta\phi
	+ \delta^{-2}|\na w_\tau|^2\na w_\tau\cdot\na\phi + w_\tau\rho_\tau\phi\big)
	dxdt\bigg| \\
	&\le \eta\|\Delta w_\tau\|_{L^2(Q_T)}\|\Delta\phi\|_{L^2(Q_T)}
	+ \eta\delta^{-2}\|\na w_\tau\|_{L^4(Q_T)}^3\|\na\phi\|_{L^4(Q_T)} \\
	&\phantom{xx}{}+ \eta\|w_\tau\rho_\tau^{1/2}\|_{L^2(Q_T)}\|\rho_\tau^{1/2}\|_{L^4(Q_T)}
	\|\phi\|_{L^4(Q_T)} \\
	&\le C(\delta)(\eta^{1/2}+\eta^{1/4})\|\phi\|_{L^4(0,T;H^2(\Omega))}
	\to 0\quad\mbox{as }\eta\to 0.
\end{align*}

Thus, performing the limit $(\eta,\tau)\to 0$ in \eqref{2.rhotau}--\eqref{2.vtau}, 
it follows that $(\rho,c)$ solve
\begin{align}
  & \int_0^T\langle\pa_t\rho,\phi\rangle dt
	+ \int_0^T\int_\Omega\big((1+\delta\rho)\na\rho - \rho\na v\big)\cdot\na\phi dxdt = 0, 
	\label{2.rhodelta} \\
	& \int_0^T\int_\Omega(\na v\cdot\na\theta + v\theta)dxdt
	= \int_0^T\int_\Omega(\delta\rho + \rho^\alpha)\theta dxdt, \label{2.vdelta}
\end{align}
where, by density, we can choose test functions $\phi\in 
L^\infty(0,T;H^1(\Omega))$ and $\theta\in L^2(0,T;$ $H^1(\Omega))$.
The initial datum $\rho(0)=\rho^0$ is satisfied in the sense of $H^1(\Omega)'$; 
see, e.g., \cite[pp.~1980f.]{Jue15} for a proof.

{\em Step 4: Limit $\delta\to 0$.} For this limit, we need further uniform estimates.
Let $(\rho_\tau,v_\tau)$ be a solution to \eqref{2.rhotau}--\eqref{2.vtau}.
We formulate \eqref{2.est2} as
\begin{align}
  \int_\Omega&(\rho_\tau^p-(\pi_\tau\rho_\tau)^p)dx
	+ \int_0^T\int_\Omega|\na\rho_\tau^{p/2}|^2 dxdt
	+ \frac{\delta}{4}\int_0^T\int_\Omega|\na \rho_\tau^{(p+1)/2}|^2 dxdt \nonumber \\
	&\le C\int_0^T\|\rho_\tau^{(p+1)/2}\|_{L^1(\Omega)}^2 dt
	+ C\int_0^T\|\rho_\tau^{p/2}\|_{L^1(\Omega)}^{\beta(p)}dt, \label{2.aux2}
\end{align}
where we recall that $C>0$ is independent of $(\eta,\tau,\delta)$. 
The $L^\infty(0,T;L^2(\Omega))$ and $L^2(0,T;$ $H^1(\Omega))$ bounds 
for $(\rho_\tau^{1/2})$ show that
\begin{align*}
  \int_0^T\|\na\rho_\tau\|_{L^1(\Omega)}^2 dt
	&= 4\int_0^T\|\rho_\tau^{1/2}\na\rho_\tau^{1/2}\|_{L^1(\Omega)}^2 dt
	\le 4\int_0^T\|\rho_\tau^{1/2}\|_{L^2(\Omega)}^2
	\|\na\rho_\tau^{1/2}\|_{L^2(\Omega)}^2 dt \\
	&\le 4\|\rho_\tau\|_{L^\infty(0,T;L^1(\Omega))}
	\int_0^T\|\na\rho_\tau^{1/2}\|_{L^2(\Omega)}^2 dt \le C.
\end{align*}
Thus, $(\rho_\tau)$ is bounded in $L^2(0,T;W^{1,1}(\Omega))\hookrightarrow
L^2(0,T;L^{3/2}(\Omega))$, as $d\le 3$.

Let $p=2$ in \eqref{2.aux2}. As the right-hand side of \eqref{2.aux2} 
is uniformly bounded, we infer the bounds
$$
  \|\rho_\tau\|_{L^\infty(0,T;L^2(\Omega))} 
  + \|\rho_\tau\|_{L^2(0,T;H^1(\Omega))} 
	+ \delta^{1/2}\|\rho_\tau^{3/2}\|_{L^2(0,T;H^1(\Omega))} \le C.
$$
Choosing $p=3$ in \eqref{2.aux2}, the right-hand side is again bounded,
yielding the estimates
$$
  \|\rho_\tau\|_{L^\infty(0,T;L^3(\Omega))} 
	+ \delta^{1/2}\|\rho_\tau^{2}\|_{L^2(0,T;H^1(\Omega))} \le C.
$$
By elliptic regularity, since $-\Delta v_\tau+v_\tau=\delta \rho_\tau+\rho_\tau^\alpha
\in L^\infty(0,T;L^3(\Omega))$, the family $(v_\tau)$ is bounded in 
$L^\infty(0,T;W^{2,3}(\Omega))\hookrightarrow L^\infty(0,T;W^{1,p}(\Omega))$ 
for all $p<\infty$.

We know from Step 3 that $(\rho_\tau,v_\tau)$ converges in some norms to
$(\rho_\delta,v_\delta):=(\rho,v)$ solving \eqref{2.rhodelta}--\eqref{2.vdelta}.
By the weakly lower semicontinuity of the norm and the a.e.\ convergence
$\rho_\tau^2\to\rho^2$ in $Q_T$, it follows that, after performing
the limit $(\eta,\tau)\to 0$,
\begin{equation}\label{2.est3}
  \|\rho_\delta\|_{L^\infty(0,T;L^3(\Omega))}
	+ \|\rho_\delta\|_{L^2(0,T;H^1(\Omega))}
	+ \delta^{1/2}\|\rho_\delta^2\|_{L^2(0,T;H^1(\Omega))} 
	+ \|v_\delta\|_{L^\infty(0,T;W^{1,p}(\Omega))} \le C.
\end{equation}

We wish to derive a uniform estimate for the time derivative $\pa_t\rho_\delta$.
Let $\phi\in L^2(0,T;H^1(\Omega))$. Then
\begin{align*}
  \bigg|\int_0^T\langle\pa_t\rho_\delta,\phi\rangle dt\bigg|
	&\le \bigg(\|\na\rho_\delta\|_{L^2(Q_T)} 
	+ \frac{\delta}{2}\|\na(\rho_\delta^2)\|_{L^2(Q_T)}\bigg)\|\na\phi\|_{L^2(Q_T)} \\
	&\phantom{xx}{}+ \|\rho_\delta\|_{L^2(0,T;L^4(\Omega))}
	\|\na v_\delta\|_{L^\infty(0,T;L^4(\Omega))}\|\na\phi\|_{L^2(Q_T)} \le C.
\end{align*}
This shows that $(\pa_t\rho_\delta)$ is bounded in $L^2(0,T;H^1(\Omega)')$.
By the Aubin--Lions lemma in the version of \cite{Sim87}, there exists a subsequence,
which is not relabeled, such that, as $\delta\to 0$,
$$
  \rho_\delta\to \rho\quad\mbox{strongly in }L^2(0,T;L^p(\Omega)), \quad
	p<6.
$$
Furthermore, we deduce from the bounds \eqref{2.est3}, again for a subsequence, that
\begin{align*}
  \na\rho_\delta\rightharpoonup\na\rho &\quad\mbox{weakly in }L^2(Q_T), \\
	\delta\na(\rho_\delta^2)\to 0 &\quad\mbox{strongly in }L^2(Q_T), \\
	\pa_t\rho_\delta\rightharpoonup\pa_t\rho &\quad\mbox{weakly in }L^2(0,T;H^1(\Omega)'),
	\\
	v_\delta\rightharpoonup^* v &\quad\mbox{weakly* in }L^\infty(0,T;W^{1,p}(\Omega)),
	\quad p<\infty.
\end{align*}
In particular, $\rho_\delta\na v_\delta\rightharpoonup\rho\na v$ weakly in
$L^2(Q_T)$. Thus, we can perform the limit $\delta\to 0$ in 
\eqref{2.rhodelta}--\eqref{2.vdelta}, which gives
\begin{align*}
  \int_0^T\langle\pa_t\rho,\phi\rangle dt
	+ \int_0^T\int_\Omega(\na\rho-\rho\na v)\cdot\na\phi dxdt &= 0, \\
  \int_0^T\int_\Omega(\na v\cdot\na\theta + v\theta)dxdt
	&= \int_0^T\int_\Omega\rho^\alpha\theta dx
\end{align*}
for all $\phi$, $\theta\in L^2(0,T;H^1(\Omega))$. Furthermore, we show as in
\cite[pp.~1980f.]{Jue15} that $\rho(0)=\rho^0$ in the sense of $H^1(\Omega)'$.

{\em Step 5: Convergence of the whole sequence.} The whole sequence
$(\rho_\delta,v_\delta)$ converges if the limit problem has a unique solution.
Uniqueness follows by standard estimates if $\rho$, $\na c\in 
L^\infty(0,T;L^\infty(\Omega))$. Since $-\Delta c+c=\rho^\alpha\in 
L^\infty(0,T;L^{3/\alpha}(\Omega))$ and $3/\alpha>3\ge d$, elliptic regularity shows that
$c\in L^\infty(0,T;W^{2,3/\alpha}(\Omega))\hookrightarrow 
L^\infty(0,T;W^{1,\infty}(\Omega))$.
Then \cite[Lemma~1]{HPS07} shows that $\rho\in L^\infty(0,T;L^\infty(\Omega))$,
finishing the proof.


\section{Proof of Theorem \ref{thm.pp}}\label{sec.pp}

It is shown in \cite{HiJu11} that \eqref{1.eq}--\eqref{1.bic} with $\alpha=1$
has a global weak
solution in two space dimensions. Since the solutions to the limiting 
Keller--Segel system may blow up after finite time, we cannot generally expect 
estimates that are uniform in $\delta$ globally in time. 
Moreover, we need higher-order estimates not provided by
the results of \cite{HiJu11}. Therefore, we show first the local existence of
smooth solutions and then uniform $H^s(\Omega)$ bounds.

{\em Step 1: Local existence of smooth solutions.}
Let $\eps=1$. The eigenvalues of the diffusion matrix associated to
\eqref{1.eq}, 
$$
  A(\rho,c) = \begin{pmatrix} 1 & -\rho \\ \delta & 1 \end{pmatrix},
$$
equal $\lambda=1\pm\mathrm{i}\sqrt{\delta\rho}$, 
and they have a positive real part for all
$\rho>0$, i.e., $A(\rho,c)$ is normally elliptic.
Therefore, according to \cite[Theorem 14.1]{Ama93} (also see \cite[Theorem 3.1]{Jue17}), 
there exists a unique maximal solution to \eqref{1.eq}--\eqref{1.bic} satisfying
$(\rho,c)\in C^\infty(\overline\Omega\times(0,T^*);\R^2)$,
where $0<T^*\le \infty$.

Next, let $\eps=0$. We use the Schauder fixed-point theorem to prove the
regularity of the solutions to \eqref{1.rhov}--\eqref{1.bic2}. 
We only sketch the proof, since the arguments are rather standard.
We introduce the set
$$
  S = \big\{\widetilde\rho\in C^0(\overline\Omega\times[0,T]):
	0\le\widetilde\rho\le R,\ 
	\|\widetilde\rho\|_{C^{\gamma,\gamma/2}(\overline\Omega\times[0,T])}\le K\big\}
$$
for some $R>0$ and $M>0$. Let $\widetilde\rho\in S$. By elliptic regularity
(combining Theorems 2.4.2.7 and 2.5.1.1 in \cite{Gri85}), the unique solution to
$$
  -\Delta v+v = \delta\widetilde\rho+\widetilde\rho^\alpha\quad\mbox{in }\Omega,
	\quad \na\widetilde\rho\cdot\nu=0\quad\mbox{on }\pa\Omega,
$$
satisfies $v\in C^0([0,T];W^{2,p}(\Omega))$ for all $p<\infty$. Hence, by
Sobolev embedding, $h:=\widetilde\rho\na v\in C^0(\overline\Omega\times[0,T])$.
Thus, using \cite[Lemma 2.1iv]{Lan17}, the unique solution to
$$
  \pa_t\rho = \diver\big((1+\delta\widetilde\rho)\na\rho - h)\quad\mbox{in }\Omega,
	\ t>0, \quad\na\rho\cdot\nu=0\quad\mbox{on }\pa\Omega,
$$
satisfies $\rho\in C^{\gamma,\gamma/2}(\overline\Omega\times[0,T])$.
By elliptic regularity again, $v\in C^{2,\gamma/2}(\overline\Omega\times[0,T])$.
Consequently, $h\in C^{\gamma,\gamma/2}(\overline\Omega\times[0,T])$
and applying \cite[Lemma 2.1iv]{Lan17} again, we infer that
$\rho\in C^{2,1}(\overline\Omega\times[0,T])$. It is possible to show that
$\rho\in S$ for suitable $R>0$ and $M>0$. Hence,
the existence of a solution to \eqref{1.rhov}--\eqref{1.bic2}
follows from the Schauder fixed-point theorem. 

Elliptic regularity implies that $v\in C^{4,1}(\overline\Omega\times[0,T])$. 
Then $f:=\diver(\rho\na v)\in C^{1,1}(\overline\Omega\times[0,T])$ and
the solution $u=\rho$ to the linear parabolic equation 
$\pa_t u-\Delta u - \diver(\rho\na u)=f$ in $\Omega$, $t>0$, with no-flux
boundary conditions satisfies $u\in 
C^{2+\gamma,1+\gamma/2}(\overline\Omega\times[0,T])$ \cite[Corollary 5.1.22]{Lun95}
(here we need $\rho^0\in C^{2+\gamma}(\overline\Omega)$).
Thus, the regularity of $f$ improves to $f\in C^{1+\gamma,1+\gamma/2}
(\overline\Omega\times[0,T])$. By parabolic regularity 
\cite[Theorem 9.2, p.~137]{Fri08}, we infer that 
$u\in C^{2+\gamma,2}(\overline\Omega\times[0,T])$. Bootstrapping this
argument and using \cite[Theorems 10.1--10.2, pp.~139f.]{Fri08}, we find that
$\rho=u\in C^{\infty}(\overline\Omega\times(0,T])$ and consequently
$v\in C^{\infty}(\overline\Omega\times(0,T])$. 

{\em Step 2: Preparations.} Let $\eps=1$, let $(\rho_\delta,c_\delta)$ be a local
smooth solution to \eqref{1.eq}-\eqref{1.bic} with $0<\delta<1$,
and let $(\rho,c)$ be a local smooth solution to \eqref{1.bic}--\eqref{1.ks}.
Then $\rho_R:=\rho_\delta-\rho$ and $c_R:=c_\delta-c$ solve
\begin{align}
  \pa_t\rho_R &= \diver\big(\na\rho_R - \rho_R\na(c+c_R) - \rho\na c_R\big), 
	\label{3.rhoR} \\
	\pa_t c_R &= \Delta c_R - c_R + \delta\Delta(\rho+\rho_R)
	+ (\rho+\rho_R)^\alpha - \rho^\alpha\quad\mbox{in }\Omega,\ t>0, \label{3.cR}
\end{align}
$(\rho_R,c_R)$ satisfies homogeneous Neumann boundary conditions and
vanishing initial conditions:
$$
  \na\rho_R\cdot\nu = \na c_R\cdot\nu = 0\quad\mbox{on }\pa\Omega,\ t>0, \quad
	\rho_R(0) = c_R(0)=0\quad\mbox{in }\Omega.
$$
The aim is to prove a differential inequality for
$$
  \Gamma(t) = \|(\rho_R,c_R)(t)\|_{H^2(\Omega)}^2, \quad
	G(t) = \|(\rho_R,c_R)(t)\|_{H^2(\Omega)}^2 
	+ \|\na\Delta(\rho_R,c_R)(t)\|_{L^2(\Omega)}^2,
$$
where $\|(\rho_R,c_R)\|^2_X = \|\rho_R\|_X^2 + \|c_R\|_X^2$ 
for suitable norms $\|\cdot\|_X$.

{\em Step 3: $H^1(\Omega)$ estimates.} We use $\rho_R$ as a test function
in \eqref{3.rhoR}:
\begin{align*}
  \frac12\frac{d}{dt}\|\rho_R\|_{L^2(\Omega)}^2
	+ \|\na\rho_R\|_{L^2(\Omega)}^2
	&= \int_\Omega\rho_R\na(c+c_R)\cdot\na\rho_R dx 
	+ \int_\Omega\rho\na c_R\cdot\na\rho_R dx \\
	&=: I_1 + I_2.
\end{align*}
By Young's inequality, for any $\eta>0$, we have
$$
  I_2 \le \frac{\eta}{2}\|\na\rho_R\|_{L^2(\Omega)}^2
	+ \frac{1}{2\eta}\|\rho\|_{L^\infty(\Omega)}^2\|\na c_R\|_{L^2(\Omega)}^2
	\le \frac{\eta}{2}\|\na\rho_R\|_{L^2(\Omega)}^2
	+ C(\eta)\|\na c_R\|_{L^2(\Omega)}^2,
$$
where here and in the following, $C>0$ and $C(\eta)>0$ denote generic constants
independent of $\delta$ but depending on suitable norms of $(\rho,c)$.
The embedding $H^2(\Omega)\hookrightarrow L^\infty(\Omega)$ (for $d\le 3$) gives
\begin{align*}
  I_1 &\le \eta\|\na\rho_R\|_{L^2(\Omega)}^2
	+ \frac{1}{2\eta}\|\na c\|_{L^\infty(\Omega)}^2\|\rho_R\|_{L^2(\Omega)}^2
	+ \frac{1}{2\eta}\|\rho_R\|_{L^\infty(\Omega)}^2\|\na c_R\|_{L^2(\Omega)}^2 \\
	&\le \eta\|\na\rho_R\|_{L^2(\Omega)}^2
	+ C(\eta)\|\rho_R\|_{L^2(\Omega)}^2 
	+ C(\eta)\|\rho_R\|_{H^2(\Omega)}^2\|\na c_R\|_{L^2(\Omega)}^2.
\end{align*}
Combining the estimates for $I_1$ and $I_2$ and choosing $\eta>0$ sufficiently
small, we find that
\begin{equation}\label{3.s1}
  \frac{d}{dt}\|\rho_R\|_{L^2(\Omega)}^2
	+ C\|\na\rho_R\|_{L^2(\Omega)}^2
	\le C\big(1+\|\rho_R\|_{H^2(\Omega)}^2\big)\|\na c_R\|_{L^2(\Omega)}^2
	+ C\|\rho_R\|_{L^2(\Omega)}^2.
\end{equation}

Next, we use the test function $c_R$ in \eqref{3.cR}:
\begin{align*}
  \frac{1}{2}\frac{d}{dt}\|c_R\|_{L^2(\Omega)}^2
	+ \|c_R\|_{H^1(\Omega)}^2 
	&= -\delta\int_\Omega\na(\rho+\rho_R)\cdot\na c_R dx 
	+ \int_\Omega\big((\rho+\rho_R)^\alpha-\rho^\alpha\big) c_Rdx \\
	&= I_3 + I_4.
\end{align*}
Thus, for any $\eta>0$,
\begin{align*}
  I_3 &\le \frac{\eta}{2}\|\na c_R\|_{L^2(\Omega)}^2
	+ \frac{\delta^2}{2\eta}\big(\|\na\rho\|_{L^2(\Omega)}^2
	+\|\na\rho_R\|_{L^2(\Omega)}^2\big) \\
	&\le \frac{\eta}{2}\|\na c_R\|_{L^2(\Omega)}^2 
	+ C(\eta)\|\na\rho_R\|_{L^2(\Omega)}^2 + C(\eta)\delta^2.
\end{align*}

For the estimate of $I_4$, we apply the mean-value theorem to the function
$s\mapsto s^\alpha$ (recalling that $\alpha\ge 1$):
$$
  |(\rho+\rho_R)^\alpha-\rho^\alpha| 
	\le C(1+\|\rho_R\|_{L^\infty(\Omega)})^{\alpha-1}|\rho_R| 
	\le C(1+\|\rho_R\|_{L^\infty(\Omega)}^{\alpha-1})|\rho_R|.
$$
Hence, together with the embedding $H^2(\Omega)\hookrightarrow L^\infty(\Omega)$,
$$
  I_4 \le \eta\|c_R\|_{L^2(\Omega)}^2 
	+ C(\eta)\big(1+\|\rho_R\|_{H^2(\Omega)}^{2(\alpha-1)}\big)
	\|\rho_R\|_{L^2(\Omega)}^2.
$$
Collecting these estimates and choosing $\eta>0$ sufficiently small, it follows that
\begin{equation}\label{3.s2}
  \frac{1}{2}\frac{d}{dt}\|c_R\|_{L^2(\Omega)}^2 + C\|c_R\|_{H^1(\Omega)}^2
	\le C\|\na\rho_R\|_{L^2(\Omega)}^2
	+ C\big(1+\|\rho_R\|_{H^2(\Omega)}^{2(\alpha-1)}\big)\|\rho_R\|_{L^2(\Omega)}^2
	+ C\delta^2.
\end{equation}

Thus, summing \eqref{3.s1} and \eqref{3.s2}, 
\begin{equation}\label{3.ineq1}
  \frac{d}{dt}\|(\rho_R,c_R)\|_{L^2(\Omega)}^2
	+ C\|(\rho_R,c_R)\|_{H^1(\Omega)}^2
	\le C\big(\Gamma(t) + \Gamma(t)^2 + \Gamma(t)^\alpha\big) + C\delta^2.
\end{equation}

{\em Step 4: $H^2(\Omega)$ estimates.} We multiply \eqref{3.rhoR} by $-\Delta\rho_R$
and integrate by parts in the expression with the time derivative:
\begin{align*}
  \frac12\frac{d}{dt}\|\na\rho_R\|_{L^2(\Omega)}^2
	+ \|\Delta\rho_R\|_{L^2(\Omega)}^2
	&= \int_\Omega\diver\big(\rho_R\na(c+c_R)\big)\Delta\rho_R dx
	+ \int_\Omega\diver(\rho\na c_R)\Delta\rho_R dx \\
	&= \int_\Omega\big(\na\rho_R\cdot\na(c+c_R)+\rho_R\Delta(c+c_R)\big)\Delta\rho_R dx \\
  &\phantom{xx}{}+ \int_\Omega\big(\na\rho\cdot\na c_R + \rho\Delta c_R)\Delta\rho_R dx
	=: I_5 + I_6.
\end{align*}
Then, taking into account inequality \eqref{a.L2} in the Appendix,
\begin{align*}
  I_5 &\le \eta\|\Delta\rho_R\|_{L^2(\Omega)}^2
	+ C(\eta)\|\na\rho_R\|_{L^2(\Omega)}^2 
	+ C(\eta)\|\na\rho_R\|_{H^1(\Omega)}^2\|\na c_R\|_{H^1(\Omega)}^2 \\
	&\phantom{xx}{}+ C(\eta)\|\rho_R\|_{L^2(\Omega)}^2
	+ C(\eta)\|\rho_R\|_{H^1(\Omega)}^2\|\Delta c_R\|_{H^1(\Omega)}^2, \\
	I_6 &\le \eta\|\Delta\rho_R\|_{L^2(\Omega)}^2
	+ C(\eta)\|\na c_R\|_{L^2(\Omega)}^2 + C(\eta)\|\Delta c_R\|_{L^2(\Omega)}^2,
\end{align*}
and choosing $\eta>0$ sufficiently small, we end up with
\begin{align}
  \frac12\frac{d}{dt}\|\na\rho_R\|_{L^2(\Omega)}^2 
	+ C\|\Delta\rho_R\|_{L^2(\Omega)}^2
	&\le C\|\rho_R\|_{H^1(\Omega)}^2 
	+ C\big(1+\|\na\rho_R\|_{H^1(\Omega)}^2\big)\|c_R\|_{H^2(\Omega)}^2 \nonumber \\
	&\phantom{xx}{}+ C\|\rho_R\|_{H^1(\Omega)}^2\|\Delta c_R\|_{H^1(\Omega)}^2.
	\label{3.s3}
\end{align}

We multiply \eqref{3.cR} by $-\Delta c_R$ and estimate similarly as in Step 3:
\begin{align*}
  \frac{1}{2}\frac{d}{dt}&\|\na c_R\|_{L^2(\Omega)}^2
	+ \|\Delta c_R\|_{L^2(\Omega)}^2 + \|\na c_R\|_{L^2(\Omega)}^2 \\
	&= \delta\int_\Omega\Delta(\rho+\rho_R)\Delta c_R dx
  + \int_\Omega\big((\rho+\rho_R)^\alpha-\rho^\alpha\big)\Delta c_R dx \nonumber \\
	&\le \eta\|\Delta c_R\|_{L^2(\Omega)}^2
	+ C\big(1+\|\rho_R\|_{H^2(\Omega)}^{2(\alpha-1)}\big)\|\rho_R\|_{L^2(\Omega)}^2
	+ C\delta^2\|\Delta\rho_R\|_{L^2(\Omega)}^2 + C\delta^2.
\end{align*}
Hence, for sufficiently small $\eta>0$,
\begin{align}
  \frac{1}{2}\frac{d}{dt}&\|\na c_R\|_{L^2(\Omega)}^2
	+ \|\Delta c_R\|_{L^2(\Omega)}^2 + \|\na c_R\|_{L^2(\Omega)}^2 \nonumber \\
	&\le C\big(1+\|\rho_R\|_{H^2(\Omega)}^{2(\alpha-1)}\big)\|\rho_R\|_{L^2(\Omega)}^2
	+ C\delta^2\|\Delta\rho_R\|_{L^2(\Omega)}^2 + C\delta^2. \label{3.cH2}
\end{align}
Adding this inequality and \eqref{3.s3}, adding $\|c_R\|_{L^2(\Omega)}^2$ on
both sides, using \eqref{a.H2} in the Appendix, and choosing $\delta>0$
sufficiently small to absorb the term $C\delta^2\|\Delta\rho_R\|_{L^2(\Omega)}^2$, 
we infer that
\begin{align}
  \frac{d}{dt}\|\na(\rho_R,c_R)\|_{L^2(\Omega)}^2
	&+ C\|(\rho_R,c_R)\|_{H^2(\Omega)}^2 
	\le C\|(\rho_R,c_R)\|_{H^2(\Omega)}^2
	+ C\|\na\rho_R\|_{H^1(\Omega)}^2\|\na c_R\|_{H^1(\Omega)}^2 \nonumber \\
	&{}+ C\|\rho_R\|_{H^1(\Omega)}^2\|\Delta c_R\|_{H^1(\Omega)}^2
	+ C\|\rho_R\|_{H^2(\Omega)}^{2(\alpha-1)}\|\rho_R\|_{L^2(\Omega)}^2 
	+ C\delta^2 \nonumber \\
	&\le C\big(\Gamma(t) + \Gamma(t)^2 + \Gamma(t)^\alpha + \Gamma(t)G(t)\big)
	+ C\delta^2. \label{3.ineq2}
\end{align}

{\em Step 5: $H^3(\Omega)$ estimates.}
We apply the Laplacian to \eqref{3.rhoR} and \eqref{3.cR} and multiply both equations
by $\Delta\rho_R$, $\Delta c_R$, respectively:
\begin{align*}
  \frac12\frac{d}{dt}&\|\Delta\rho_R\|_{L^2(\Omega)}^2
	+ \|\na\Delta\rho_R\|_{L^2(\Omega)}^2 
	= \int_\Omega\na\diver(\rho_R\na(c+c_R))\cdot\na\Delta\rho_R dx \\
	&{}+ \int_\Omega\na\diver(\rho\na c_R)\cdot\na\Delta\rho_R dx =: I_7 + I_8, \\
	\frac{1}{2}\frac{d}{dt}&\|\Delta c_R\|_{L^2(\Omega)}^2
	+ \|\na\Delta c_R\|_{L^2(\Omega)}^2 + \|\Delta c_R\|_{L^2(\Omega)}^2 
  = \delta\int_\Omega\na\Delta(\rho+\rho_R)\cdot\na\Delta c_R dx \\
	&{}+ \int_\Omega\na\big((\rho+\rho_R)^\alpha-\rho^\alpha\big)\cdot\na\Delta c_R dx
	=: I_9 + I_{10}.
\end{align*}
We estimate
\begin{align*}
  I_8 &\le \eta\|\na\Delta\rho_R\|_{L^2(\Omega)}^2
	+ C(\eta)\|\na\diver(\rho\na c_R)\|_{L^2(\Omega)}^2
	+ C(\eta)\|\rho_R\|_{H^2(\Omega)}^2 \\
	&\le \eta\|\na\Delta\rho_R\|_{L^2(\Omega)}^2
	+ C(\eta)\|c_R\|_{H^2(\Omega)}^2
	+ C(\eta)\|\na\Delta c_R\|_{L^2(\Omega)}^2
	+ C(\eta)\|\rho_R\|_{H^2(\Omega)}^2.
\end{align*}
Taking into account
\begin{align*}
  \na\diver(\rho_R\na c_R)
	&= \na\big(\na\rho_R\cdot\na c_R + \rho_R\Delta c_R\big) \\
	&= (\na c_R\cdot\na)\na\rho_R + (\na\rho_R\cdot\na)\na c_R
	+ \na\rho_R\Delta c_R + \rho_R\na\Delta c_R,
\end{align*}
and inequalities \eqref{a.L2} and \eqref{a.H3} in the Appendix
as well as the embedding $H^2(\Omega)\hookrightarrow L^\infty(\Omega)$, we obtain
\begin{align*}
  \|\na\diver(\rho_R\na c_R)\|_{L^2(\Omega)}^2
	&\le \|\na c_R\|_{L^\infty(\Omega)}^2\|\na^2\rho_R\|_{L^2(\Omega)}^2
	+ \|\na\rho_R\|_{H^1(\Omega)}^2\|\na^2 c_R\|_{H^1(\Omega)}^2 \\
	&\phantom{xx}{}+ \|\na\rho_R\|_{H^1(\Omega)}^2\|\Delta c_R\|_{H^1(\Omega)}^2
	+ \|\rho_R\|_{L^\infty(\Omega)}^2\|\na\Delta c_R\|_{L^2(\Omega)}^2 \\
	&\le C\|\rho_R\|_{H^2(\Omega)}^2\big(\|\na\Delta c_R\|_{L^2(\Omega)}^2
	+ \|c_R\|_{H^2(\Omega)}^2\big),
\end{align*}
which gives
$$
  I_7 \le \eta\|\na\Delta c_R\|_{L^2(\Omega)}^2
	+ C(\eta)\|\rho_R\|_{H^2(\Omega)}^2 
	+ C(\eta)\|\rho_R\|_{H^2(\Omega)}^2\big(\|\na\Delta c_R\|_{L^2(\Omega)}^2
	+ \|c_R\|_{H^2(\Omega)}^2\big).
$$
This shows that, again for sufficiently small $\eta>0$,
\begin{align}
  \frac12\frac{d}{dt}\|\Delta\rho_R\|_{L^2(\Omega)}^2
	+ C\|\na\Delta\rho_R\|_{L^2(\Omega)}^2 
  &\le C\|(\rho_R,c_R)\|_{H^2(\Omega)}^2 + C\big(1+\|\rho_R\|_{H^2(\Omega)}^2\big)
	\|\na\Delta c_R\|_{L^2(\Omega)}^2 \nonumber \\
	&\phantom{xx}{}+ C\|\rho_R\|_{H^2(\Omega)}^2\|c_R\|_{H^2(\Omega)}^2. \label{3.s4}
\end{align}

Next, we estimate $I_9$ and $I_{10}$:
\begin{align*}
  I_9 &\le \eta\|\na\Delta c_R\|_{L^2(\Omega)}^2 
	+ C(\eta)\delta^2\|\na\Delta\rho_R\|_{L^2(\Omega)}^2 + C\delta^2, \\
  I_{10} &= \alpha\int_\Omega\big((\rho+\rho_R)^{\alpha-1}-\rho^{\alpha-1})
	\na\rho\cdot\na\Delta c_R dx \\
	&\phantom{xx}{}
	+ \alpha\int_\Omega(\rho+\rho_R)^{\alpha-1}\na\rho_R\cdot\na\Delta c_R dx 
	= J_1 + J_2.
\end{align*}
We find that
$$
  J_2 \le \eta\|\na\Delta c_R\|_{l^2(\Omega)}^2
	+ C(\eta)\big(1+\|\rho_R\|_{H^2(\Omega)}^{2(\alpha-1)}\big)
	\|\na\rho_R\|_{L^2(\Omega)}^2.
$$
By the H\"older continuity of $s\mapsto s^{\alpha-1}$, it follows that
\begin{align*}
  J_2 &\le C\int_\Omega|\rho_R|^{\alpha-1}|\na\Delta c_R|dx
	\le \eta\|\na\Delta c_R\|_{L^2(\Omega)}^2
	+ C(\eta)\|\rho_R\|_{L^{2(\alpha-1)}(\Omega)}^{2(\alpha-1)} \\
	&\le \eta\|\na\Delta c_R\|_{L^2(\Omega)}^2
	+ C(\eta)\|\rho_R\|_{H^2(\Omega)}^{2(\alpha-1)}.
\end{align*}
Consequently, for sufficiently small $\eta>0$,
\begin{align}
  \frac{1}{2}\frac{d}{dt}&\|\Delta c_R\|_{L^2(\Omega)}^2
	+ C\|\na\Delta c_R\|_{L^2(\Omega)}^2 + \|\Delta c_R\|_{L^2(\Omega)}^2 \nonumber \\
	&\le C\|\rho_R\|_{H^2(\Omega)}^{2(\alpha-1)}
	+ C\big(1+\|\rho_R\|_{H^2(\Omega)}^{2(\alpha-1)}\big)\|\na\rho_R\|_{L^2(\Omega)}^2
	+ C\delta^2\|\na\Delta\rho_R\|_{L^2(\Omega)}^2 + C\delta^2. \label{3.cH3}
\end{align}

Adding the previous inequality and \eqref{3.s4} and taking $\delta>0$ sufficiently
small such that the term $C\delta^2\|\na\Delta\rho_R\|_{L^2(\Omega)}^2$ is absorbed
by the corresponding term on the left-hand side of \eqref{3.s4}, we infer that
\begin{align}
  \frac{d}{dt}&\|(\rho_R,c_R)\|_{L^2(\Omega)}^2
	+ C\|\na\Delta(\rho_R,c_R)\|_{L^2(\Omega)}^2 + C\|\Delta c_R\|_{L^2(\Omega)}^2 
	\nonumber \\
  &\le C\|(\rho_R,c_R)\|_{H^2(\Omega)}^2	+ C\big(1+\|\rho_R\|_{H^2(\Omega)}^2\big)
	\|\na\Delta c_R\|_{L^2(\Omega)}^2 
	+ C\|\rho_R\|_{H^2(\Omega)}^2\|c_R\|_{H^2(\Omega)}^2 \nonumber \\
	&\phantom{xx}{}+ C\|\rho_R\|_{H^2(\Omega)}^{2(\alpha-1)}
	+ C\big(1+\|\rho_R\|_{H^2(\Omega)}^{2(\alpha-1)}\big)\|\na\rho_R\|_{L^2(\Omega)}^2
	+ C\delta^2 \nonumber \\
	&\le C\big(\Gamma(t) + \Gamma(t)^{\alpha-1} + \Gamma(t)^2 + \Gamma(t)^\alpha
	+ \Gamma(t)G(t)\big) + C\delta^2. \label{3.ineq3}
\end{align}

{\em Step 6: End of the proof for $\eps=1$.}
We sum inequalities \eqref{3.ineq1}, \eqref{3.ineq2}, and \eqref{3.ineq3}:
\begin{align}
  \frac{d}{dt}\big(&\|(\rho_R,c_R)\|_{H^1(\Omega)}^2
	+ \|\Delta(\rho_R,c_R)\|_{L^2(\Omega)}^2\big) \nonumber \\
	&\phantom{xx}{}+ C\big(\|(\rho_R,c_R)\|_{H^2(\Omega)}^2
	+ \|\na\Delta(\rho_R,c_R)\|_{L^2(\Omega)}^2\big) \nonumber \\
	&\le C\big(\Gamma(t) + \Gamma(t)^{\alpha-1} + \Gamma(t)^2 + \Gamma(t)^\alpha
	+ \Gamma(t)G(t)\big) + C\delta^2. \label{3.s5}
\end{align}
To get rid of the term $\Gamma(t)^{\alpha-1}$, we need the condition $\alpha\ge 2$.
Indeed, under this condition,
$$
  \Gamma(t)^{\alpha-1} \le \Gamma(t) + \Gamma(t)^\alpha.
$$
We also remove the term $\Gamma(t)^2$ by defining 
$\kappa:= \max\{\alpha,2\}$ and estimating
$$
  \Gamma(t)^2 \le \Gamma(t) + \Gamma(t)^{\kappa}.
$$
We deduce from elliptic regularity that
$$
  \Gamma(t) \le C\|\Delta(\rho_R,c_R)\|_{L^2(\Omega)}^2
	+ C\|(\rho_R,c_R)\|_{H^1(\Omega)}^2.
$$
Therefore, integrating \eqref{3.s5} over $(0,t)$ and observing that
$(\rho_R,c_R)(0)=0$, \eqref{3.s5} becomes
$$
  \Gamma(t) + C\int_0^t G(s)ds \le C\int_0^t(\Gamma(s)+\Gamma(s)^{\kappa})ds
	+ C\int_0^t \Gamma(s)G(s)ds + C\delta^2.
$$
Lemma \ref{lem.gronwall} proves the result for $\eps=1$.

{\em Step 7: Parabolic-elliptic case $\eps=0$.}
Since there is no time derivative of $c_R$ anymore, we need to change the
definition of the functionals $\Gamma(t)$ and $G(t)$:
$$
  \Gamma_0(t) = \|\rho_R\|_{H^2(\Omega)}^2, \quad
	G_0(t) = \|\rho_R\|_{H^2(\Omega)}^2 + \|\na\Delta\rho_R\|_{L^2(\Omega)}^2.
$$
The estimates are very similar to the parabolic-parabolic case with two exceptions:
In \eqref{3.ineq3}, we have estimated the terms $\|\rho_R\|_{H^2(\Omega)}^2
\|\na\Delta c_R\|_{L^2(\Omega)}^2$ and
$\|\rho_R\|_{H^2(\Omega)}^2\|c_R\|_{H^2(\Omega)}^{2}$ from above by $\Gamma(t)G(t)$.
In the present case, we cannot estimate $\|\na\Delta c_R\|_{L^2(\Omega)}^2$ 
by $G_0(t)$ and we need to proceed in a different way.

Estimates \eqref{3.cH2} and \eqref{3.cH3}, adapted to the case $\eps=0$, become
\begin{align*}
  \|\Delta c_R\|_{L^2(\Omega)}^2 + \|\na c_R\|_{L^2(\Omega)}^2
	&\le C\big(1+\|\rho_R\|_{H^2(\Omega)}^{2(\alpha-1)}\big)\|\rho_R\|_{L^2(\Omega)}^2
	+ C\delta^2\|\Delta\rho_R\|_{L^2(\Omega)}^2 + C\delta^2 \\
	&\le C\big(\Gamma_0(t) + \Gamma_0(t)^\alpha\big) + C\delta^2, \\
  \|\na\Delta c_R\|_{L^2(\Omega)}^2 + \|\Delta c_R\|_{L^2(\Omega)}^2 
 	&\le C\|\rho_R\|_{H^2(\Omega)}^{2(\alpha-1)}
	+ C\big(1+\|\rho_R\|_{H^2(\Omega)}^{2(\alpha-1)}\big)\|\na\rho_R\|_{L^2(\Omega)}^2 \\
	&\phantom{xx}{}+ C\delta^2\|\na\Delta\rho_R\|_{L^2(\Omega)}^2 + C\delta^2.
\end{align*}
The term $\delta^2\|\na\Delta\rho_R\|_{L^2(\Omega)}^2$ can be absorbed
by the corresponding term on the left-hand side of \eqref{3.s4}. The
critical term $\|\rho_R\|_{H^2(\Omega)}^{2(\alpha-1)}\|\na\rho_R\|_{L^2(\Omega)}^2$
is bounded from above by $\Gamma(t)^\alpha$. Thus, $\|\na\Delta c_R\|_{L^2(\Omega)}^2$
is estimated by $\Gamma_0(t)^\alpha$ and lower-order terms, and consequently,
$\|\rho_R\|_{H^2(\Omega)}^2$ $\times\|\na\Delta c_R\|_{L^2(\Omega)}^2$ in \eqref{3.s4}
is estimated by $\Gamma_0(t)^\alpha G_0(t)$, together with lower-order terms. 
Furthermore, $\|\rho_R\|_{H^2(\Omega)}^2\|c_R\|_{H^2(\Omega)}^{2}$ is bounded by
$\Gamma_0(t)^{\alpha+1}$, up to lower-order terms. More precisely,
a computation shows that
\begin{align*}
  \frac{d}{dt}&\big(\|\rho_R\|_{H^1(\Omega)}^2 + \|\Delta\rho_R\|_{L^2(\Omega)}^2\big)
	+ C\big(\|\rho_R\|_{H^2(\Omega)}^2 + \|\na\Delta\rho_R\|_{L^2(\Omega)}^2\big)
	+ CG_0(t) \\
  &\le C\big(\Gamma_0(t) + \Gamma_0(t)^2 + \Gamma_0(t)^{\alpha-1} + \Gamma_0(t)^\alpha
	+ \Gamma_0(t)^{\alpha+1} + \Gamma_0(t)G_0(t) \\
	&\phantom{xx}{}+ \Gamma_0(t)^\alpha G_0(t)
	+ \delta^2 G_0(t)\big) + C\delta^2.
\end{align*}
Observing that 
$\Gamma_0(t)\le C(\|\rho_R\|_{H^1(\Omega)}^2 + \|\Delta\rho_R\|_{L^2(\Omega)}^2)$
and $\Gamma_0(t)^{\alpha-1}+\Gamma_0(t)^\alpha\le \Gamma_0(t)+\Gamma_0(t)^{\alpha+1}$, 
choosing $\delta>0$ sufficiently small, integrating in time, and using
$\Gamma_0(0)=0$, we arrive at
\begin{align*}
  \Gamma_0(t) + C\int_0^t G_0(s)ds 
	&\le C\int_0^t(\Gamma_0(s)+\Gamma_0(t)^{\alpha+1})ds \\
	&\phantom{xx}{}+ C\int_0^t(\Gamma_0(s)+\Gamma_0(s)^{\alpha+1})G_0(s)ds + C\delta^2.
\end{align*}
An application of Lemma \ref{lem.gronwall} finishes the proof.


\section{Numerical experiments}\label{sec.num}

We present some numerical examples for system \eqref{1.eq}--\eqref{1.bic} 
in two space dimensions and for various choices of the parameters 
$\alpha$ and $\delta$. Equations \eqref{1.eq} are discretized by the implicit
Euler method in time and by cubic finite elements in space. 
The scheme is implemented by using the finite-element library NGSolve/Netgen 
(http://ngsolve.org). The mesh is refined in regions where large gradients
are expected. The number of vertices is between 2805 and 12,448, and the number
of elements is between 5500 and 24,030. The time step is chosen between
$10^{-3}$ and $10^{-4}$ when no blow up is expected and is decreased down to
$10^{-13}$ close to expected blow-up times.
The resulting nonlinear discrete systems are solved by
the standard Newton method. The Jacobi matrix is computed by the NGSolve routine
{\tt AssembleLinearization}. The surface plots are generated by the Python package
Matplotlib \cite{Hun07}. We do not use any kind of additional regularizations,
smoothing tools, or slope-limiters. All experiments are performed for the
parabolic-parabolic equations with $\delta>0$.

We choose the same domain and initial conditions as in \cite{CHJ12}, i.e.
$\Omega=\{x\in\R^2:|x|<1\}$ and 
\begin{equation}\label{4.rho0}
  \rho^0(x,y) = 80(x^2+y^2-1)^2(x-0.1)^2 + 5, \quad
	c^0(x,y) = 0, \quad (x,y)\in\Omega.
\end{equation}
A computation shows that the total mass $M=\int_\Omega\rho^0 dx=25\pi/3>8\pi$
is supercritical, i.e., the solution to the classical Keller--Segel system
can blow up in the interior of the domain. A sufficient condition is that
the initial density is suffíciently concentrated in the sense that
$\int_\Omega|x-x_0|^2\rho^0 dx$ is sufficiently small for some $x_0\in\Omega$.
Blow up at the boundary can occur if $x_0\in\pa\Omega$ and $M>4\pi$.

{\em Experiment 1: $\alpha=1$.} We choose the initial datum \eqref{4.rho0}
and the values $\alpha=1$, $\delta=10^{-3}$. In this nonsymmetric setting, 
the solution exists for all time and the density is expected to concentrate 
at the boundary \cite{HiJu11}.
Figure \ref{fig.alpha1} shows the surface plots for the cell density at various times.
Since the total mass is initially concentrated near the boundary, we observe
a boundary peak. Observe that there is no $L^\infty$ blow-up. 
The steady state is reached at approximately $T=2.5$.
By Theorem \ref{thm.pp}, the peak approximates the blow-up solution to the
classical Keller--Segel system in the $L^\infty$ norm;
see Figure \ref{fig.delta1}.
We see that the $L^\infty$ norm of the density becomes larger with decreasing values
of $\delta$.

\begin{figure}[ht]
\includegraphics[width=80mm]{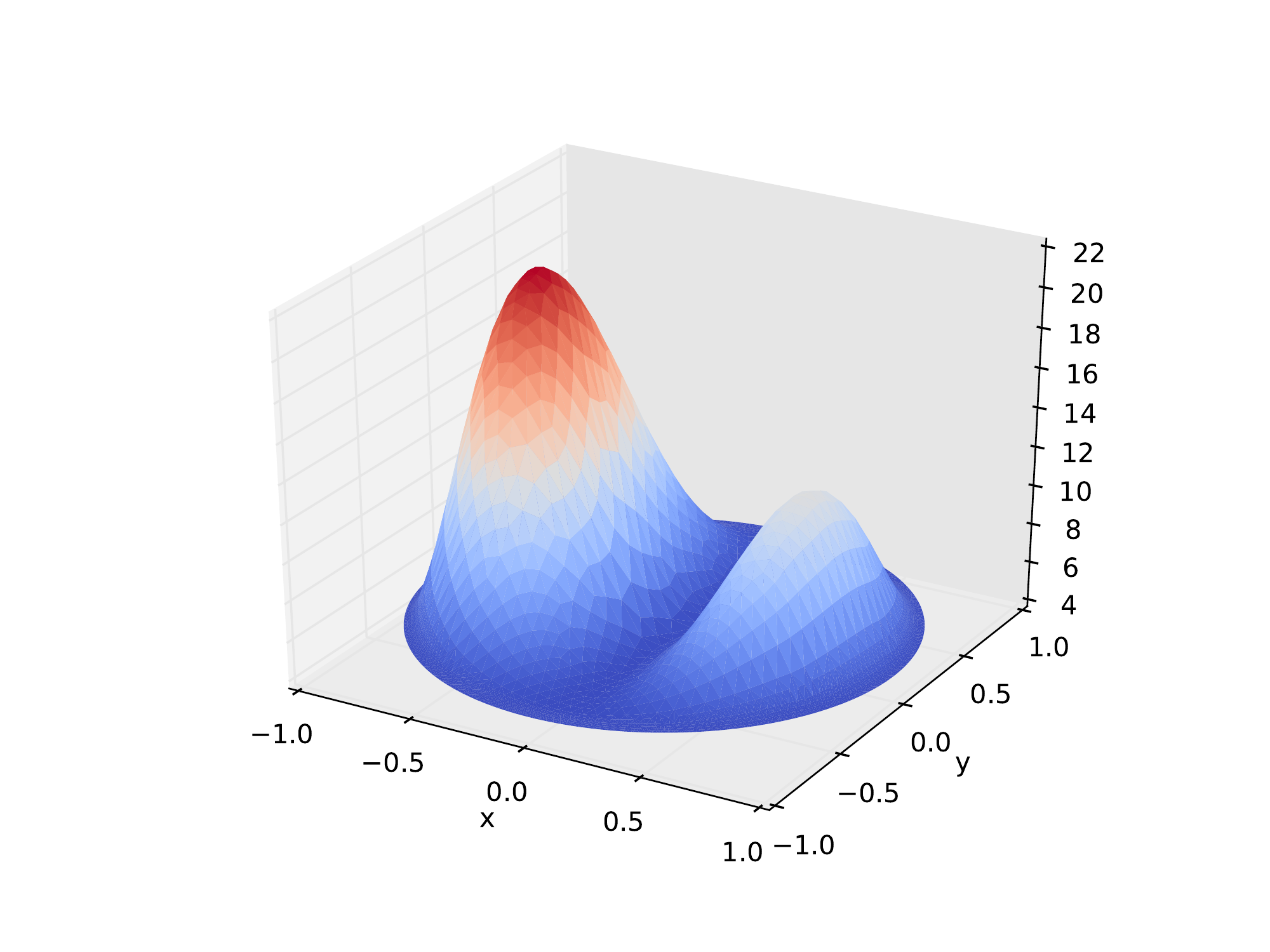}
\includegraphics[width=80mm]{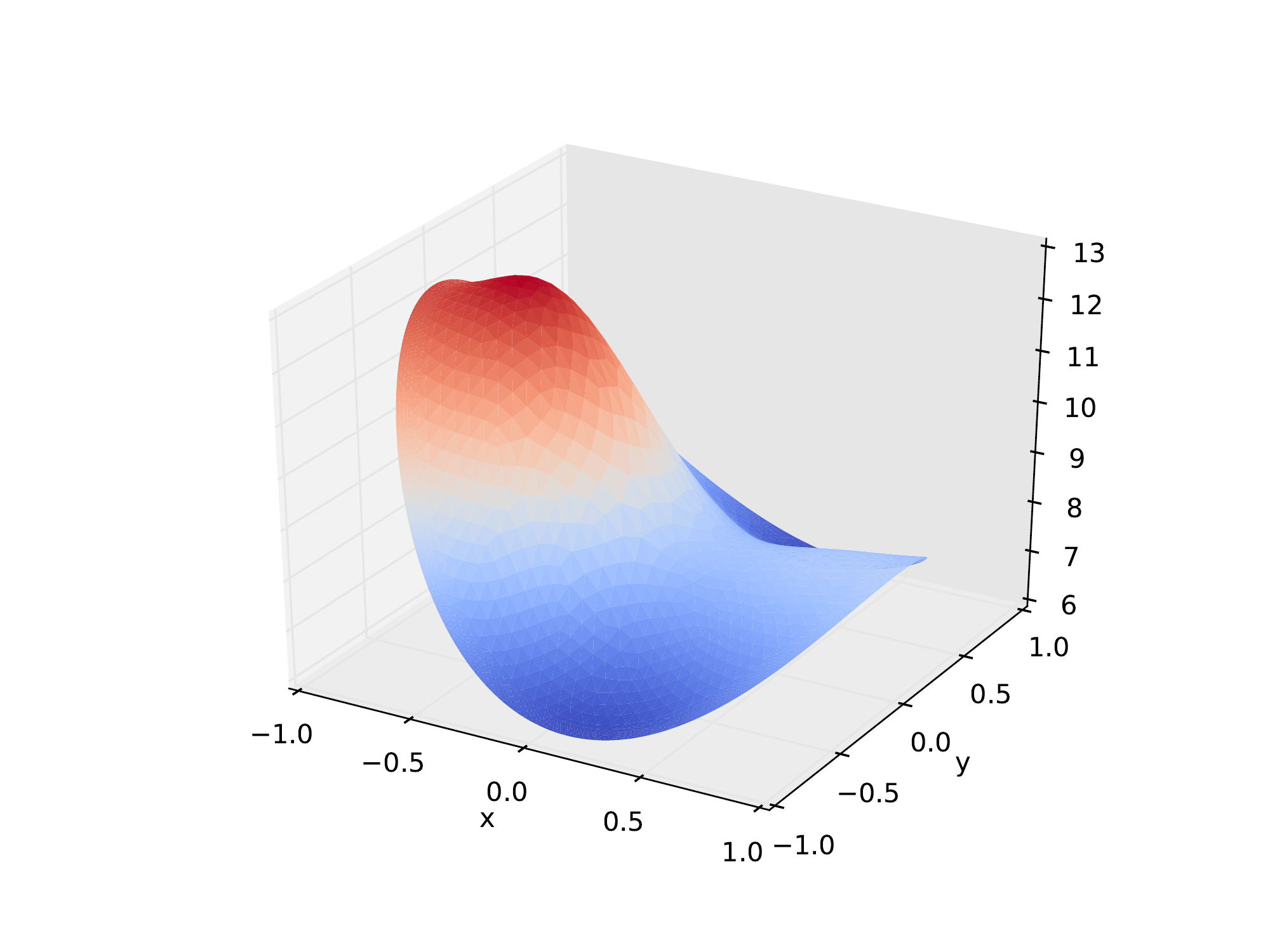}
\includegraphics[width=80mm]{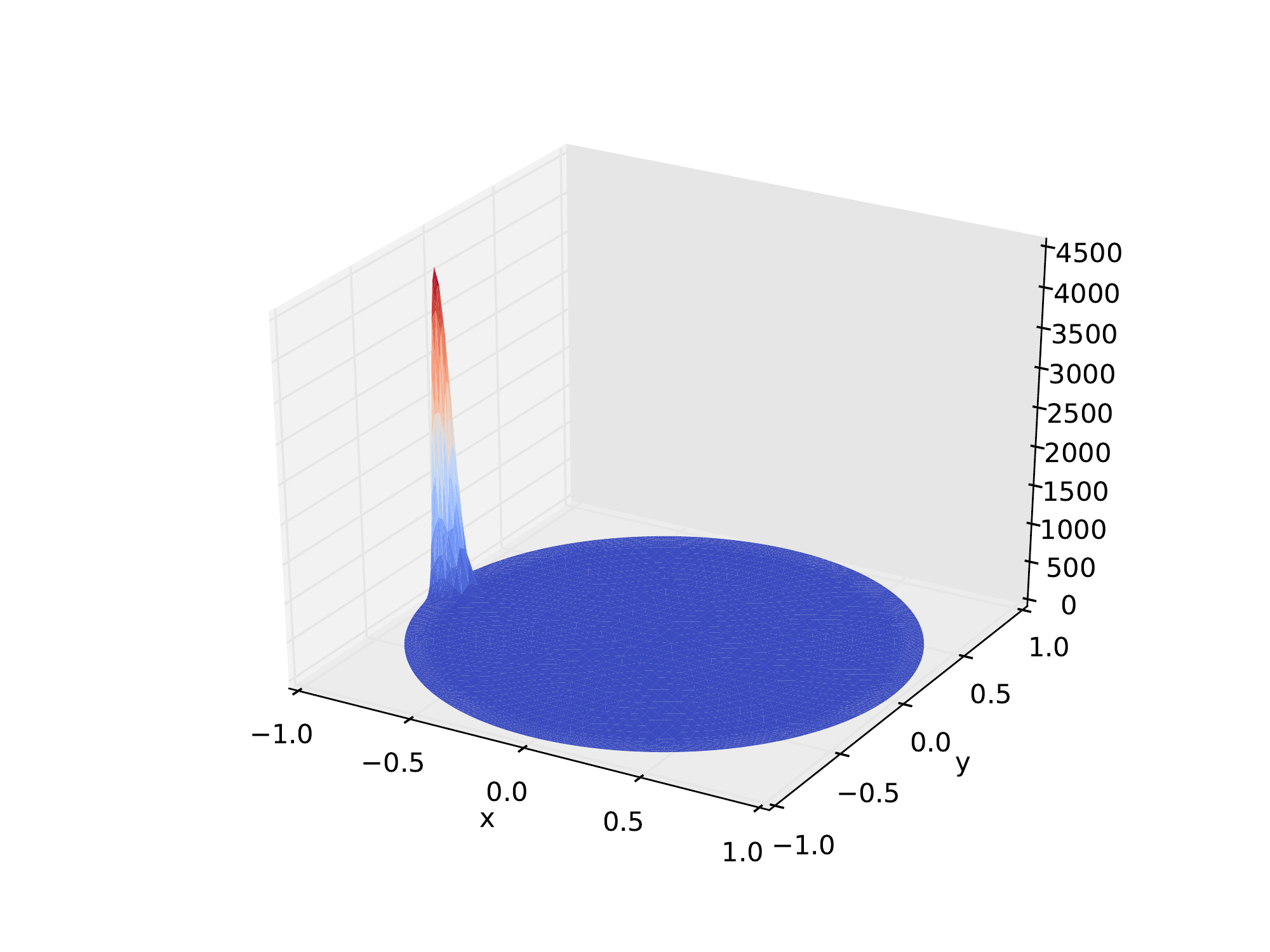}
\includegraphics[width=80mm]{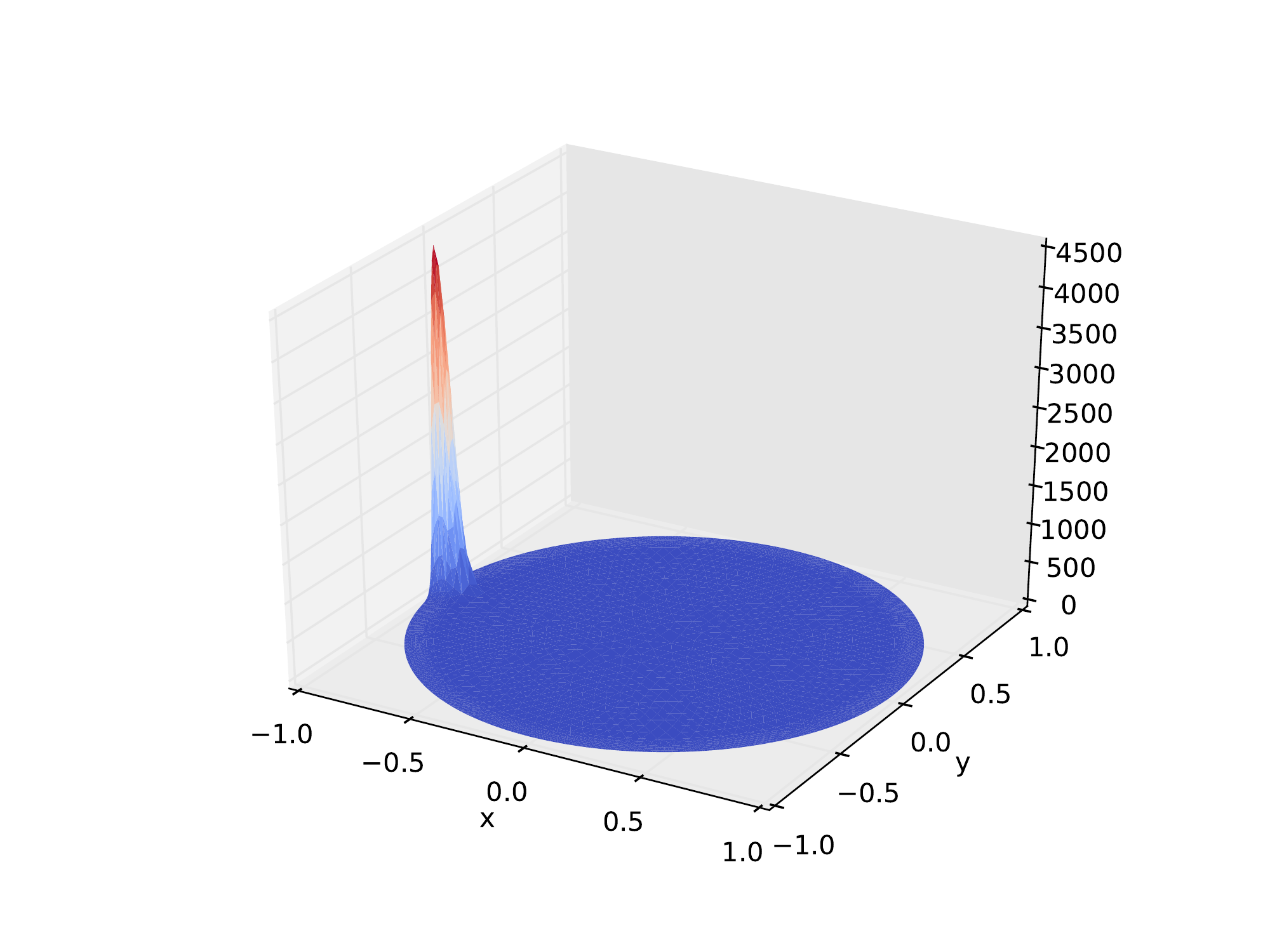}
\caption{Cell density with $\alpha=1$ and $\delta=10^{-3}$
at times $t=0$ (top left), $t=0.1$ (top right), $t=2.0$ (bottom left), 
$t=5.0$ (bottom right).}
\label{fig.alpha1}
\end{figure}

\begin{figure}[ht]
\includegraphics[width=80mm]{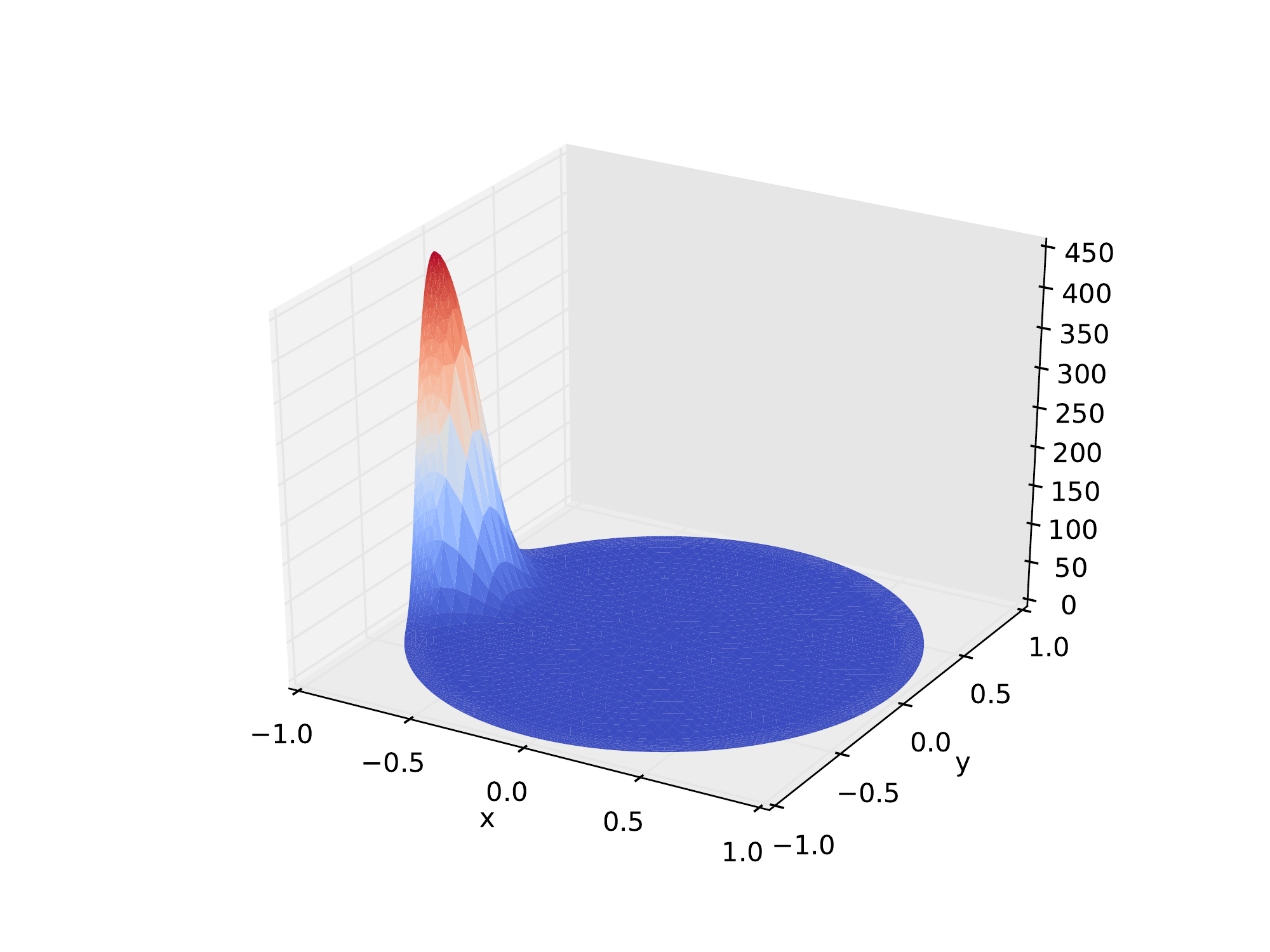}
\includegraphics[width=80mm]{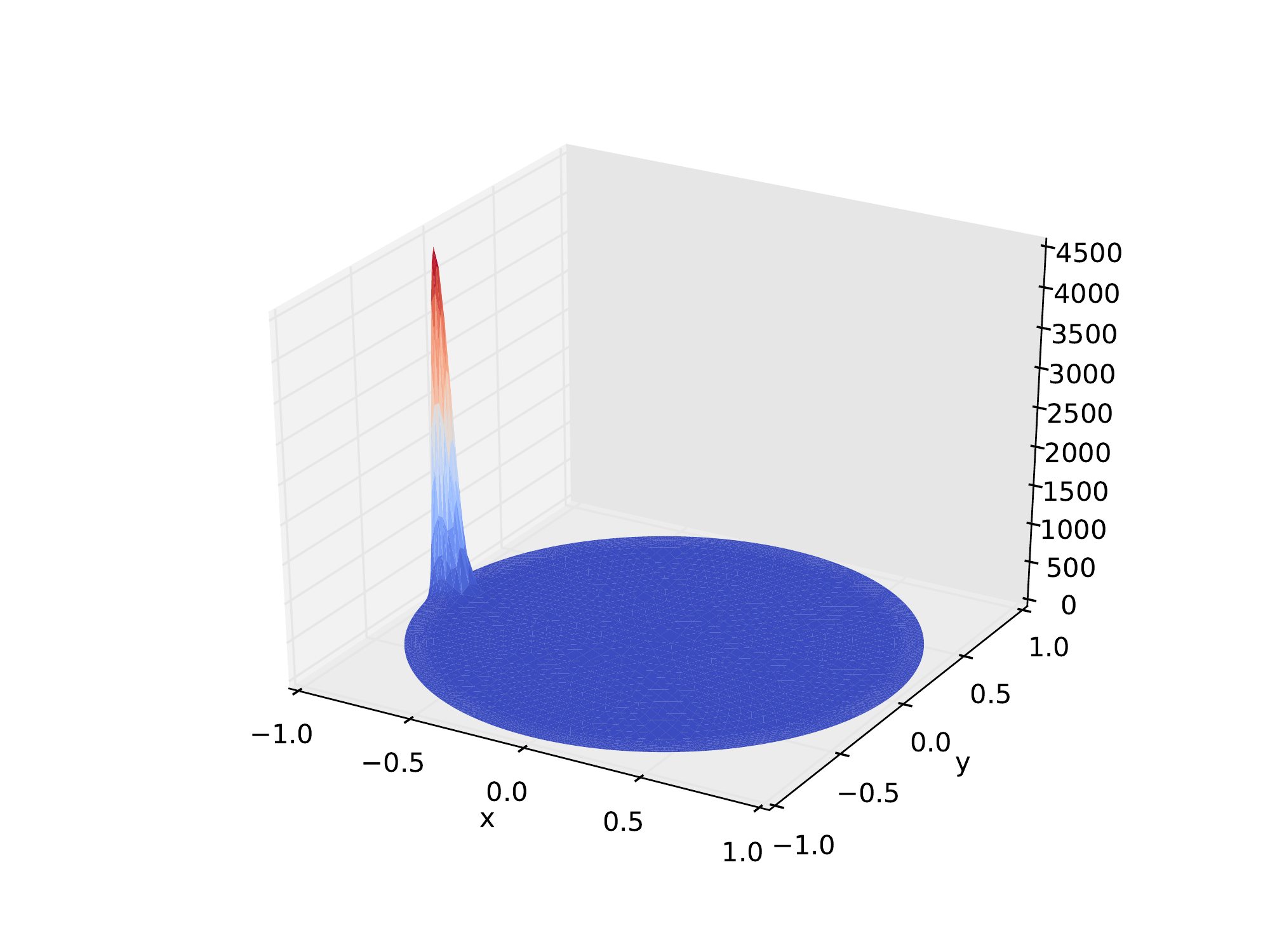}
\includegraphics[width=80mm]{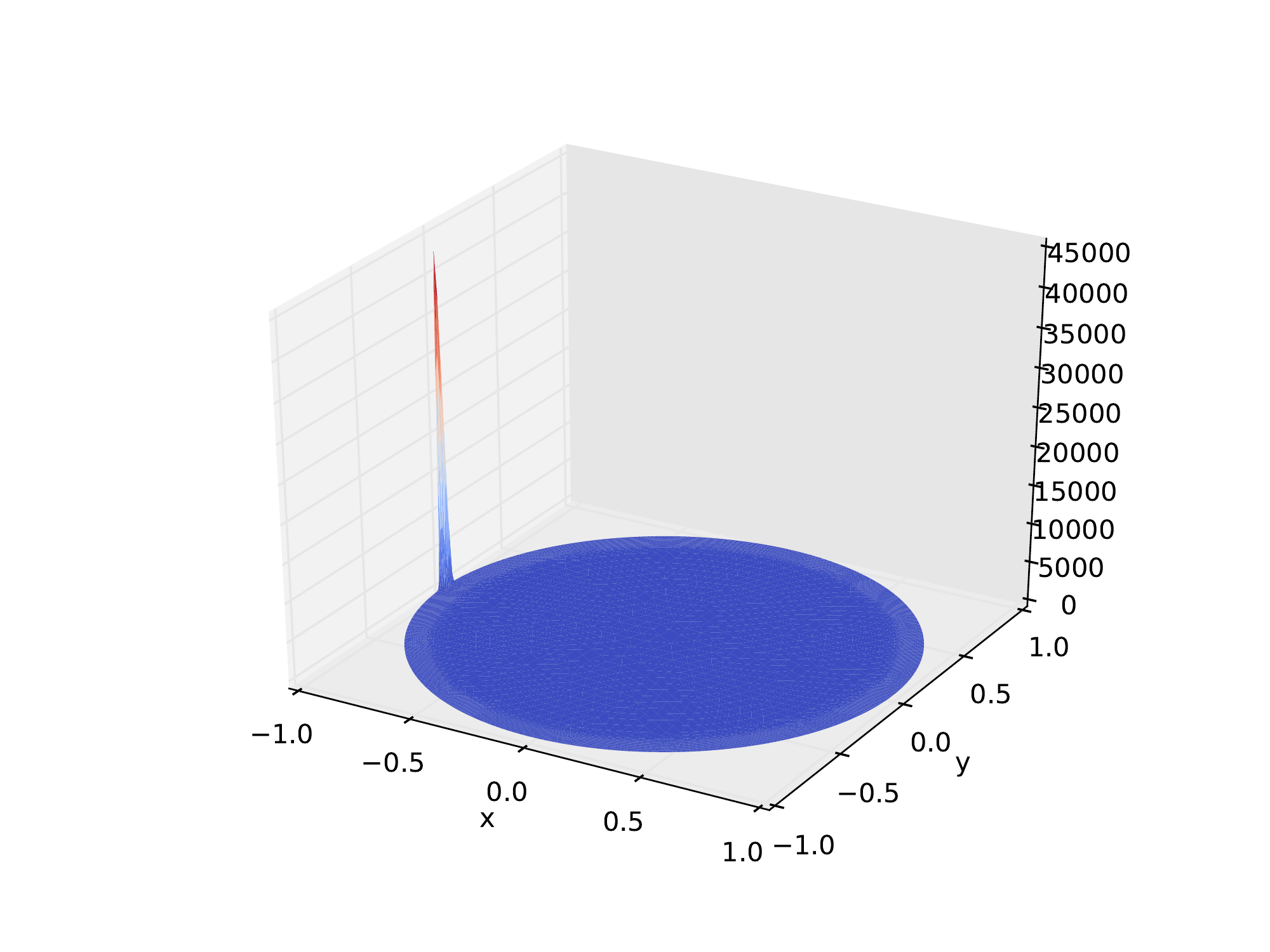}
\includegraphics[width=75mm]{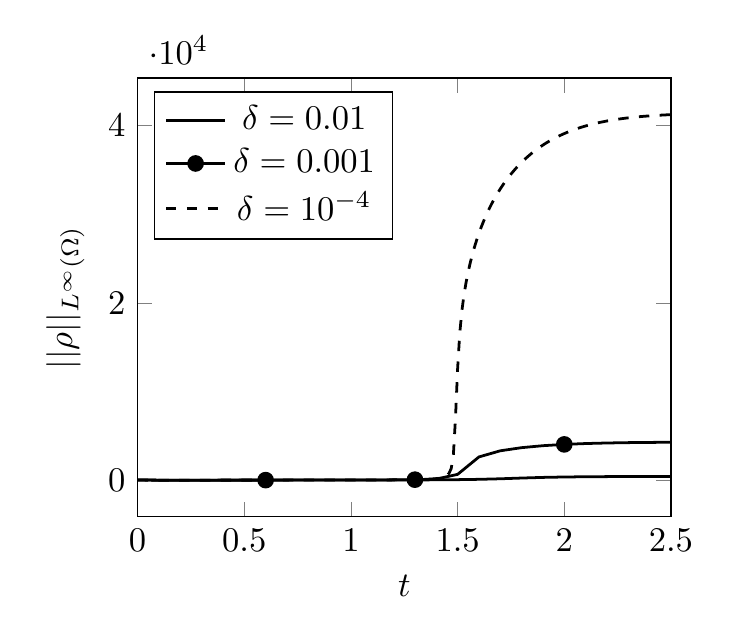}
\caption{Cell density at time $T=2.5$ with $\alpha=1$ and 
$\delta=10^{-2}$ (top left), $\delta=10^{-3}$ (top right), $\delta=10^{-4}$
(bottom left). The $L^\infty$ norm of the density is shown in the bottom right panel.}
\label{fig.delta1}
\end{figure}

{\em Experiment 2: $\alpha>1$.} First, we choose
the value $\alpha=1.5$. The initial datum is still given by \eqref{4.rho0}. 
Since $\alpha>1$, we cannot exclude finite-time blow-up, which is
confirmed by the numerical experiments in Figure \ref{fig.alpha15}. Numerically,
the solution seems to exist until time $T^*\approx 0.079$. The numerical scheme
breaks down at slightly smaller times when $\delta$ becomes smaller. This may indicate
that the numerical break-down is an upper bound for the blow-up time of the
classical Keller--Segel model. The break-down time becomes smaller for larger
values of $\alpha$. Indeed, Figure \ref{fig.alpha25} shows a stronger and faster
concentration behavior when we take $\alpha=2.5$.

\begin{figure}[ht]
\includegraphics[width=80mm]{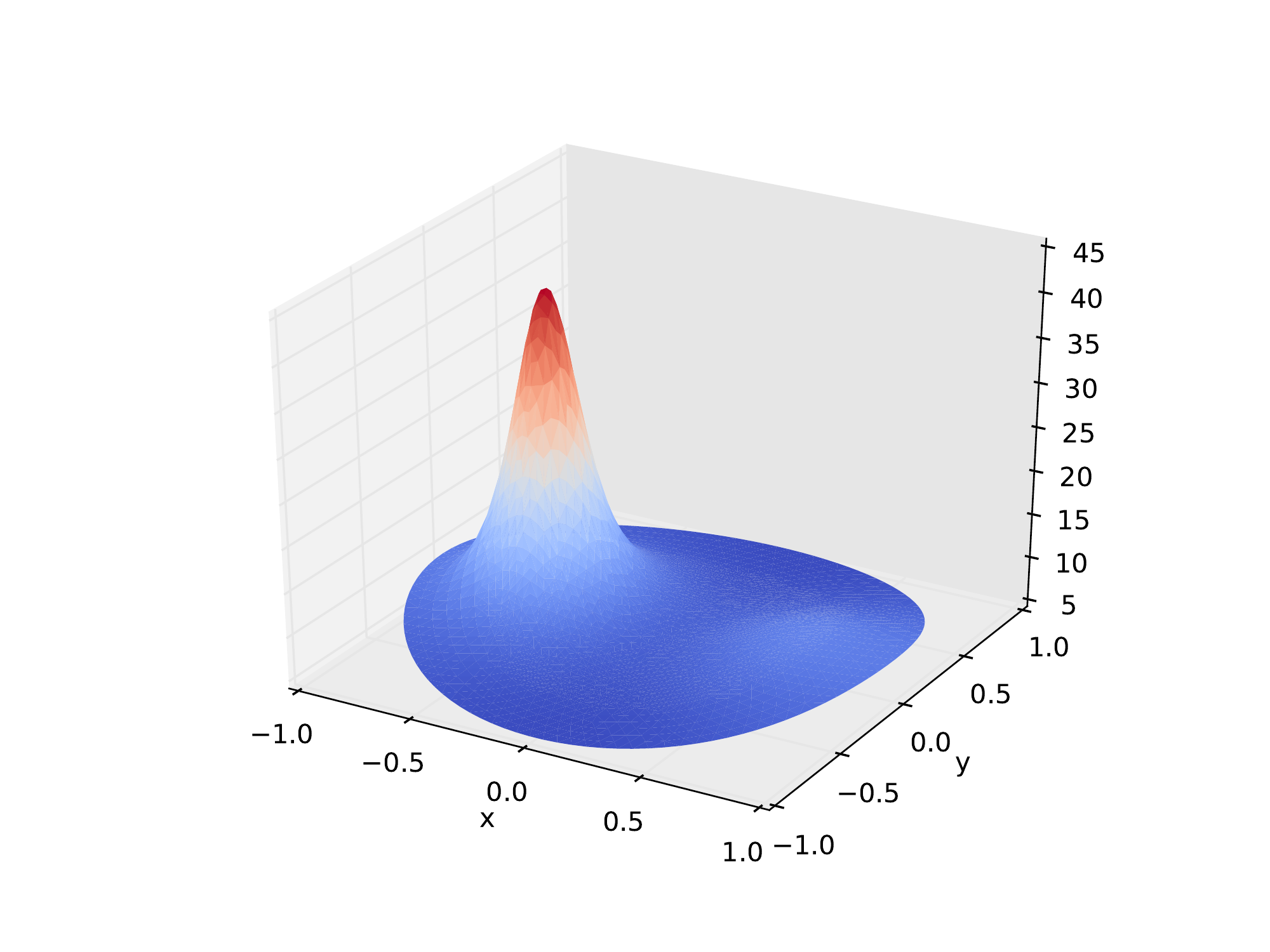}
\includegraphics[width=80mm]{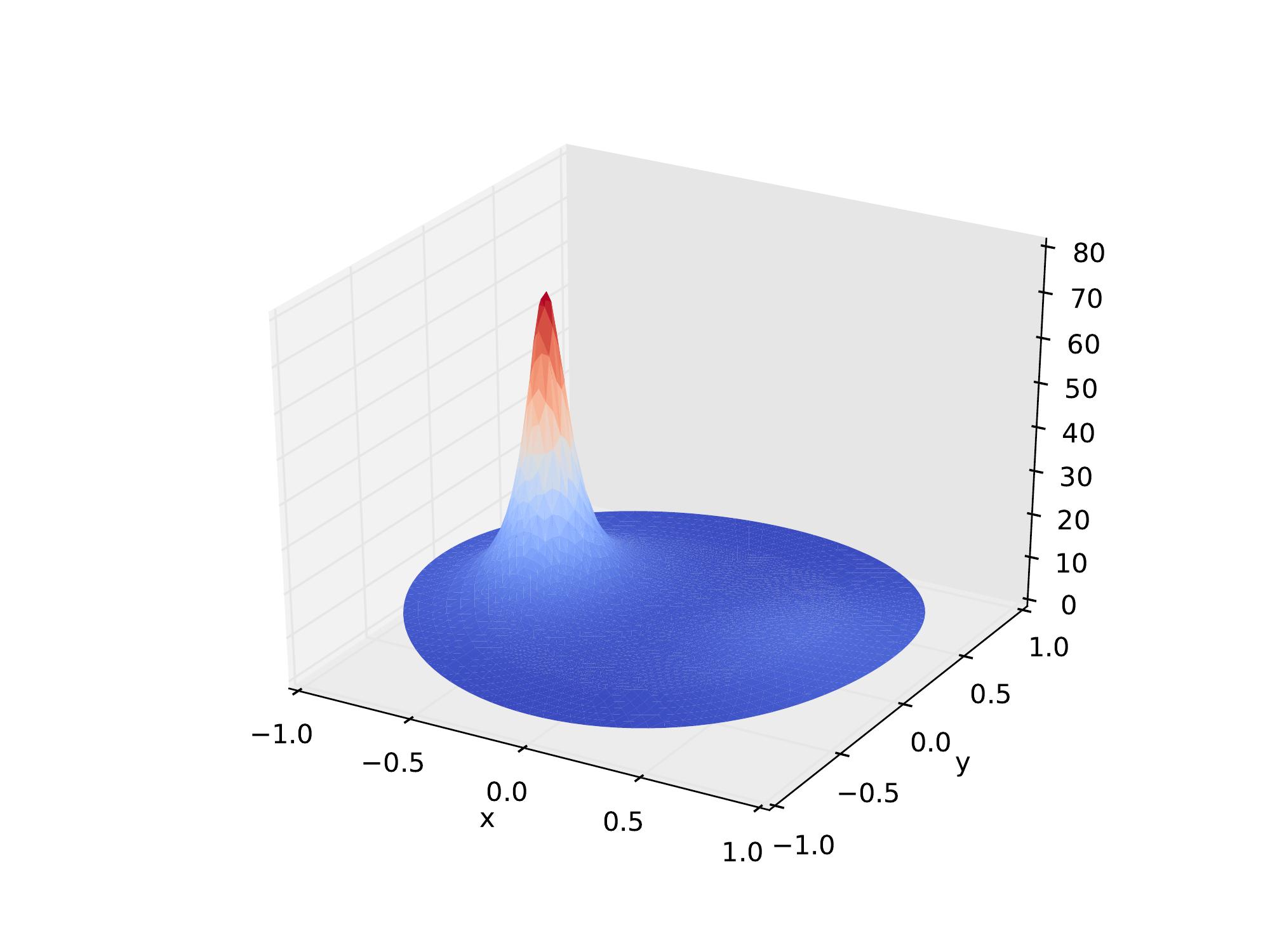}
\includegraphics[width=80mm]{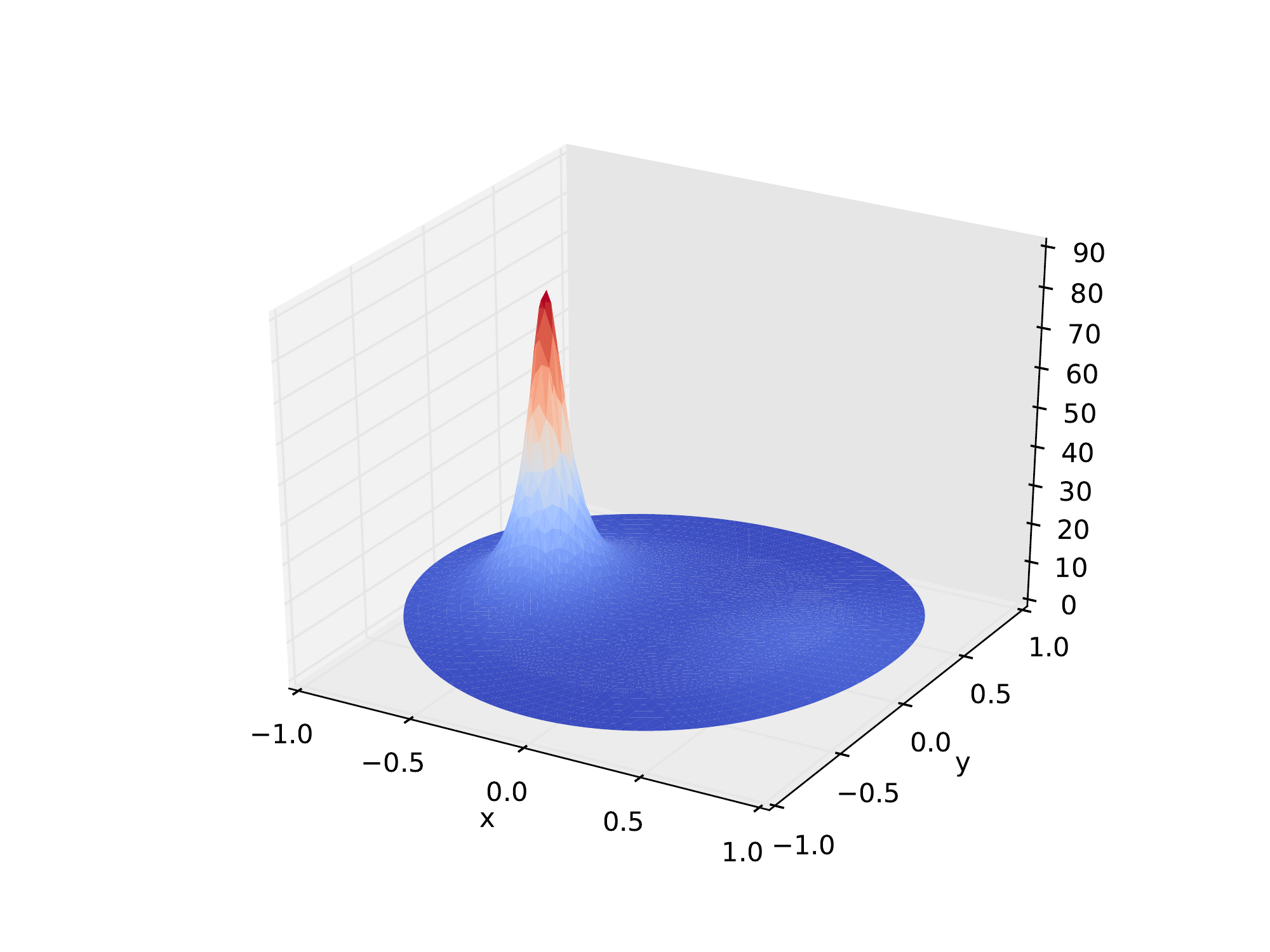}
\includegraphics[width=75mm]{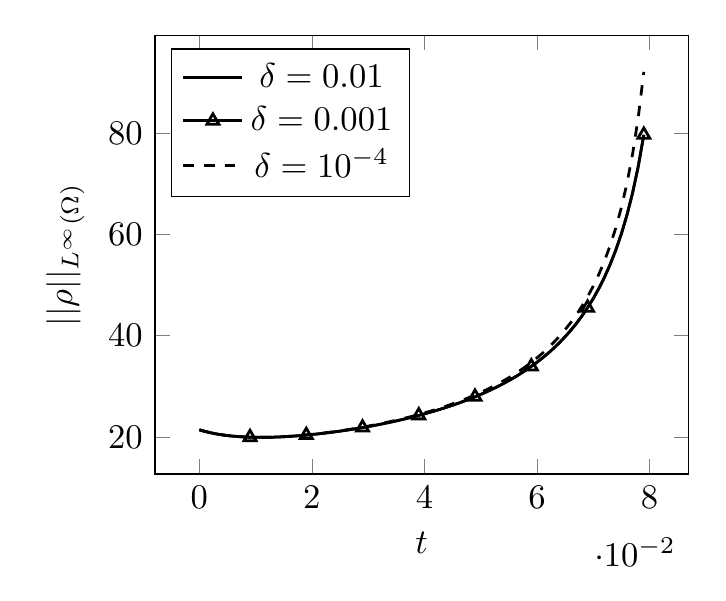}
\caption{Cell density at time $T^*=0.079$ with $\alpha=1.5$ and 
$\delta=10^{-2}$ (top left), $\delta=10^{-3}$ (top right), $\delta=10^{-4}$
(bottom left). The $L^\infty$ norm of the density is shown in the bottom right panel.}
\label{fig.alpha15}
\end{figure}

\begin{figure}[ht]
\includegraphics[width=80mm]{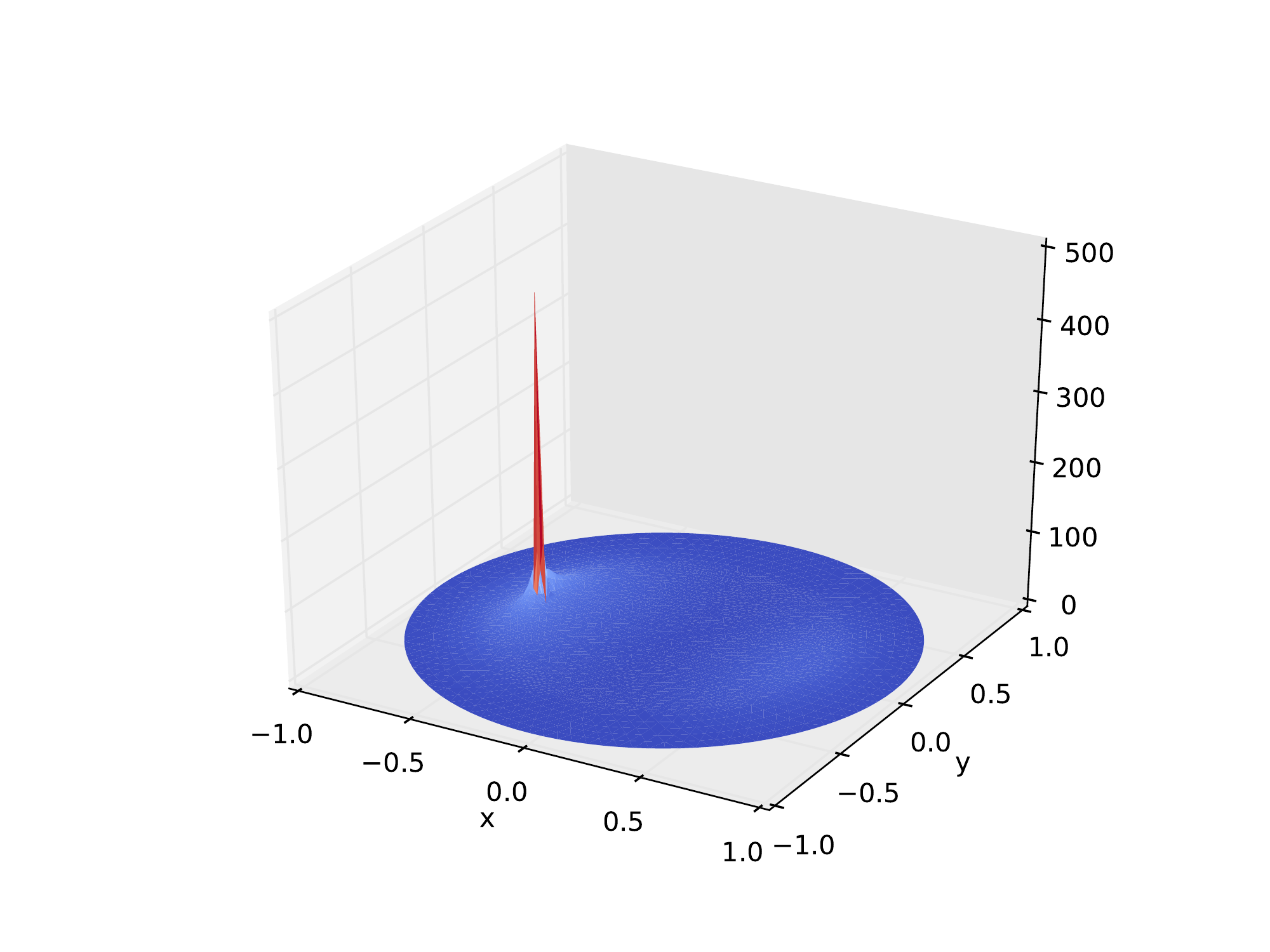}
\includegraphics[width=80mm]{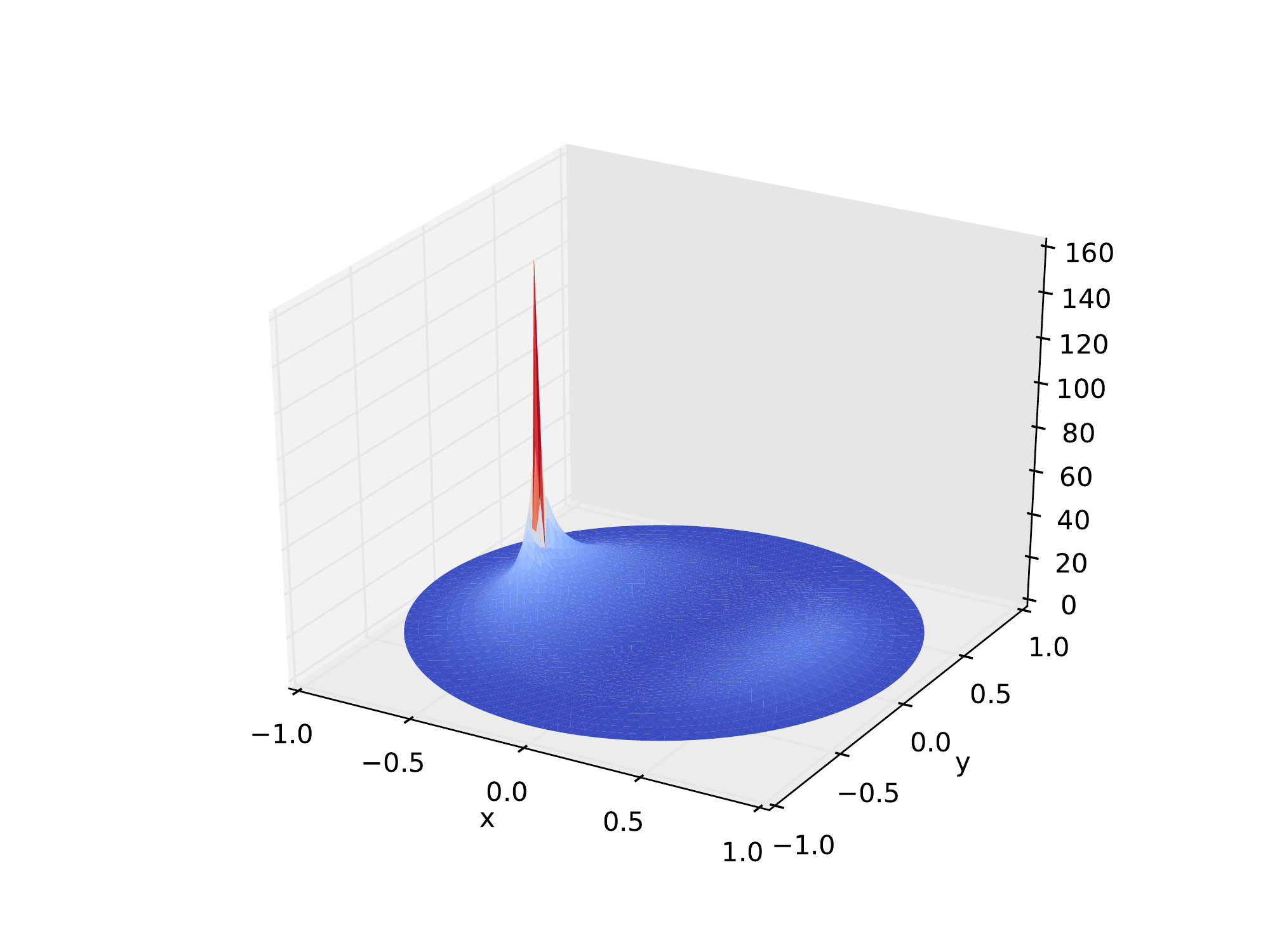}
\includegraphics[width=80mm]{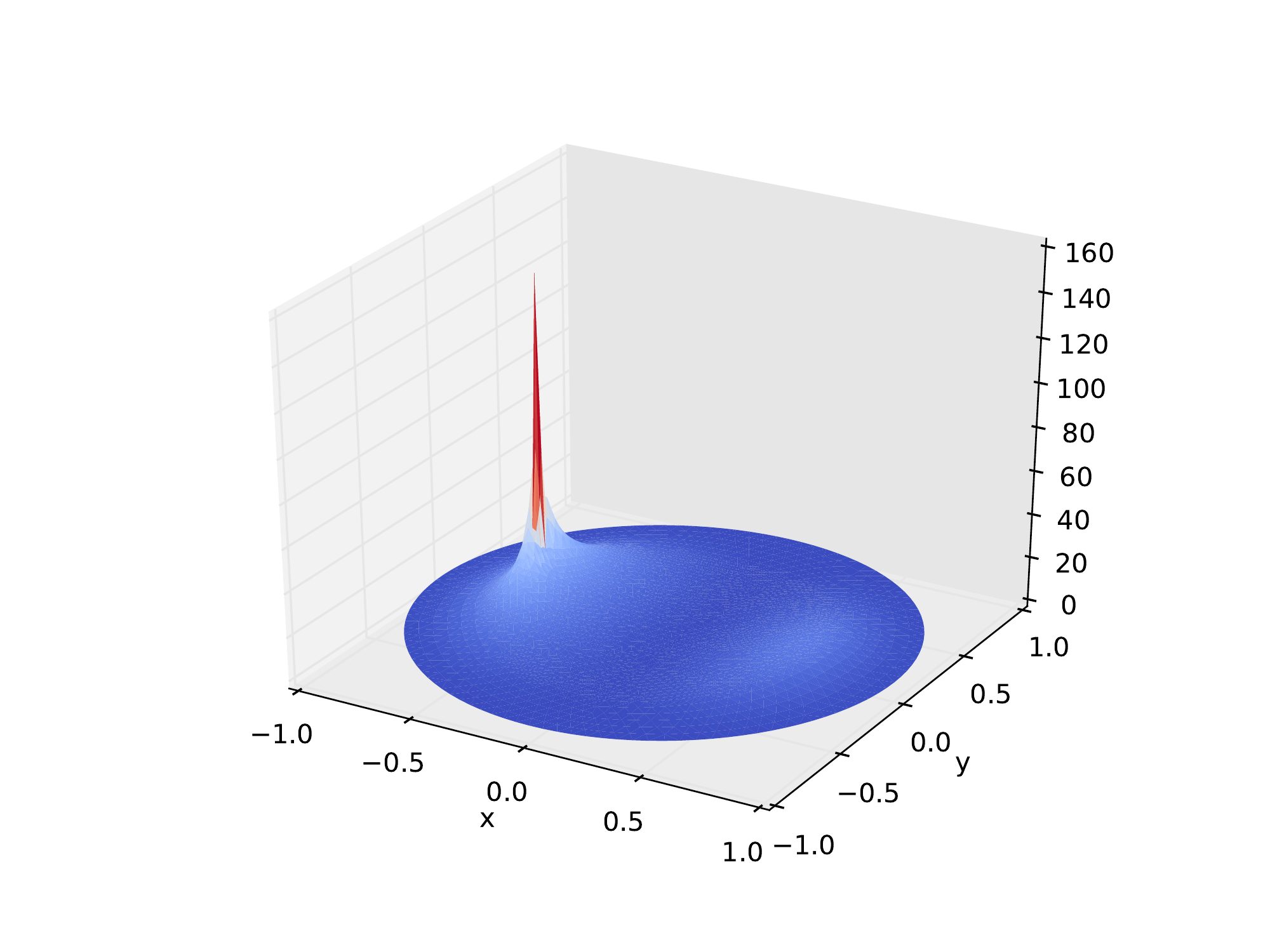}
\includegraphics[width=75mm]{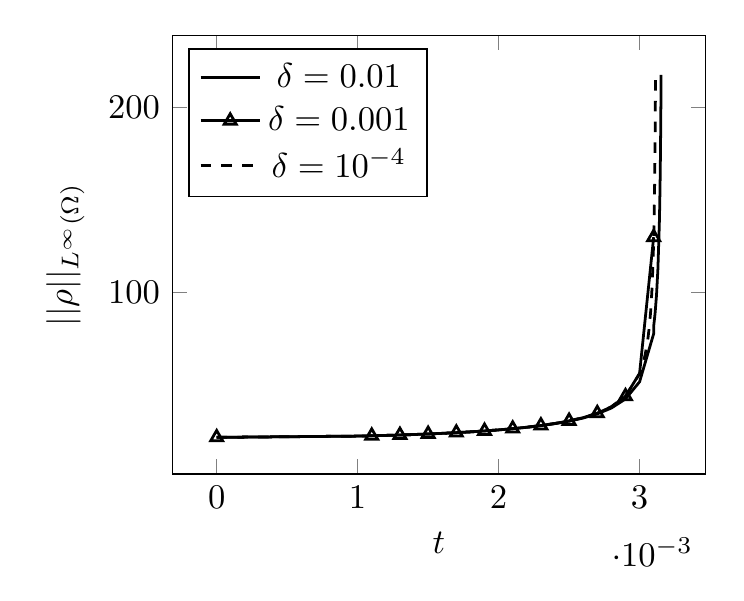}
\caption{Cell density at time $T^*=3.35\cdot 10^{-3}$ with $\alpha=2.5$ 
and $\delta=10^{-2}$ (top left), $\delta=10^{-3}$ (top right), $\delta=10^{-4}$
(bottom left). The $L^\infty$ norm of the density is shown in the bottom right panel.}
\label{fig.alpha25}
\end{figure}

{\em Experiment 3: Multi-bump initial datum.} We take $\alpha=1$ and
$\delta=5\cdot 10^{-3}$. As initial datum, we choose a linear combination of
the bump function
$$
  W_{x_0,y_0,M}(x,y) = \frac{M}{2\pi\theta}\exp\bigg(
	-\frac{(x-x_0)^2+(y-y_0)^2}{2\theta}\bigg), \quad(x,y)\in\Omega,
$$
where $(x_0,y_0)\in\Omega$, $M>0$, and $\theta>0$. Setting $\theta=10^{-2}$,
we define $c^0=0$ and
\begin{align*}
  \rho^0 &= W_{0.25,0,10\pi} + W_{-0.25,0,4\pi} + W_{0,-0.25,4\pi} + W_{0,0.25,4\pi} \\
	&\phantom{xx}{}+ W_{0,0.5,4\pi} + W_{0,0.35,4\pi} + W_{0.5,0,4\pi} 
	+ W_{0.5,0.25,4\pi}.
\end{align*}
The evolution of the density is presented in Figure \ref{fig.exp3}. The density
concentrates in the interior of the domain and the peak travels to the boundary.
At time $t=1$, the peak is close to the boundary which is reached later at $t=2.5$
(not shown). A similar behavior was already mentioned in \cite{BeJu13} for the
parabolic-elliptic model using a single-bump initial datum.

\begin{figure}[ht]
\includegraphics[width=80mm]{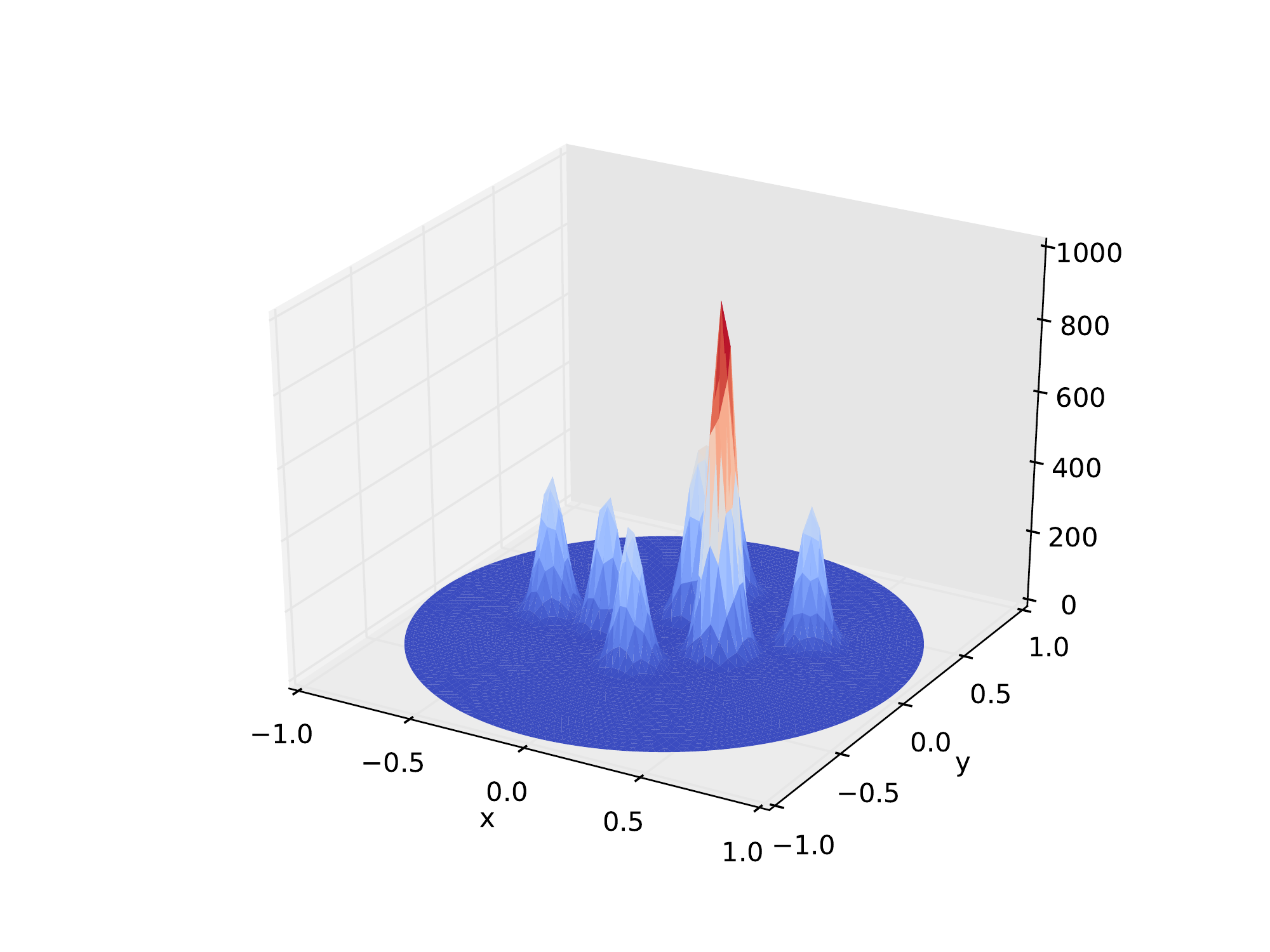}
\includegraphics[width=80mm]{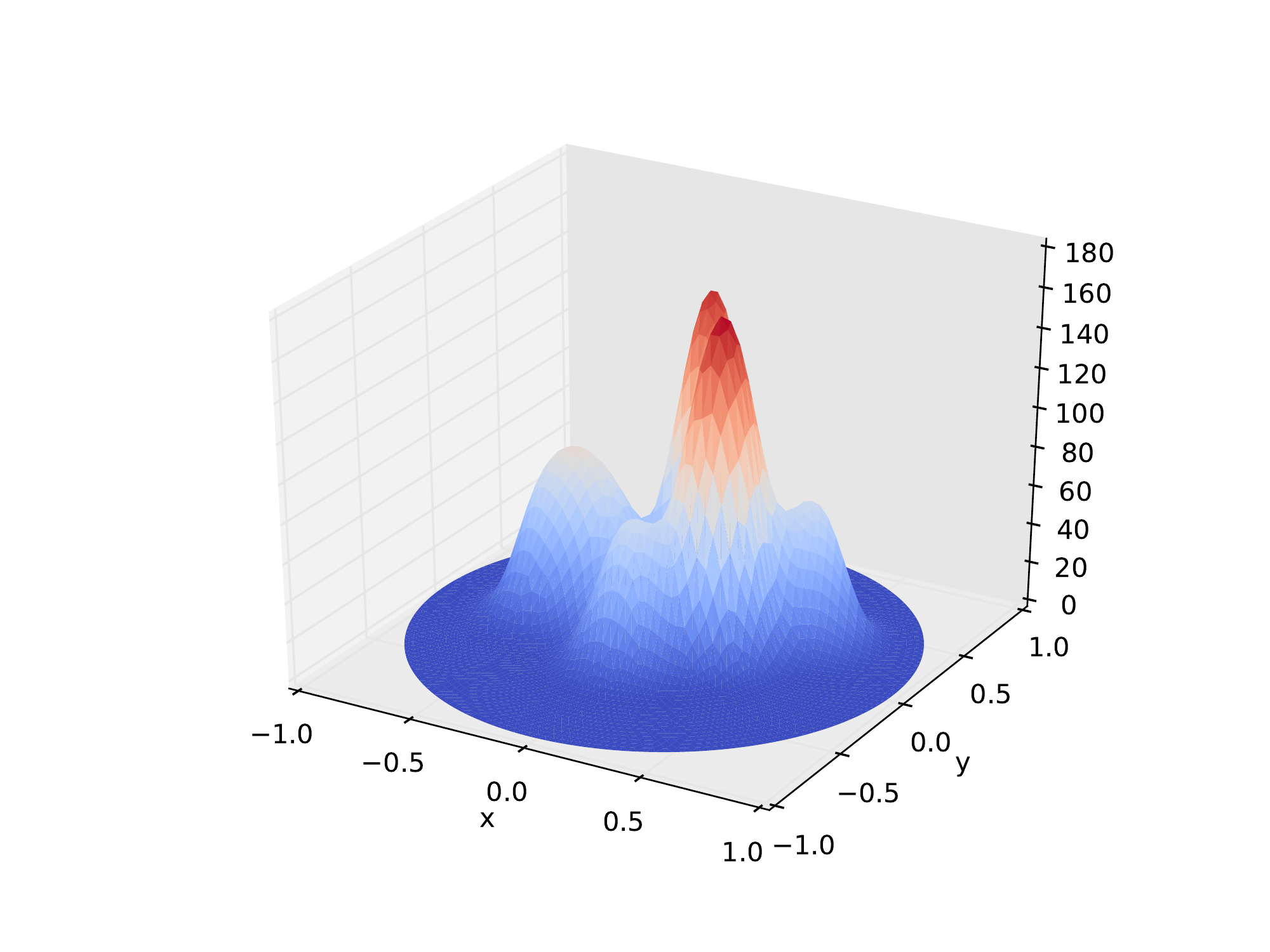}
\includegraphics[width=80mm]{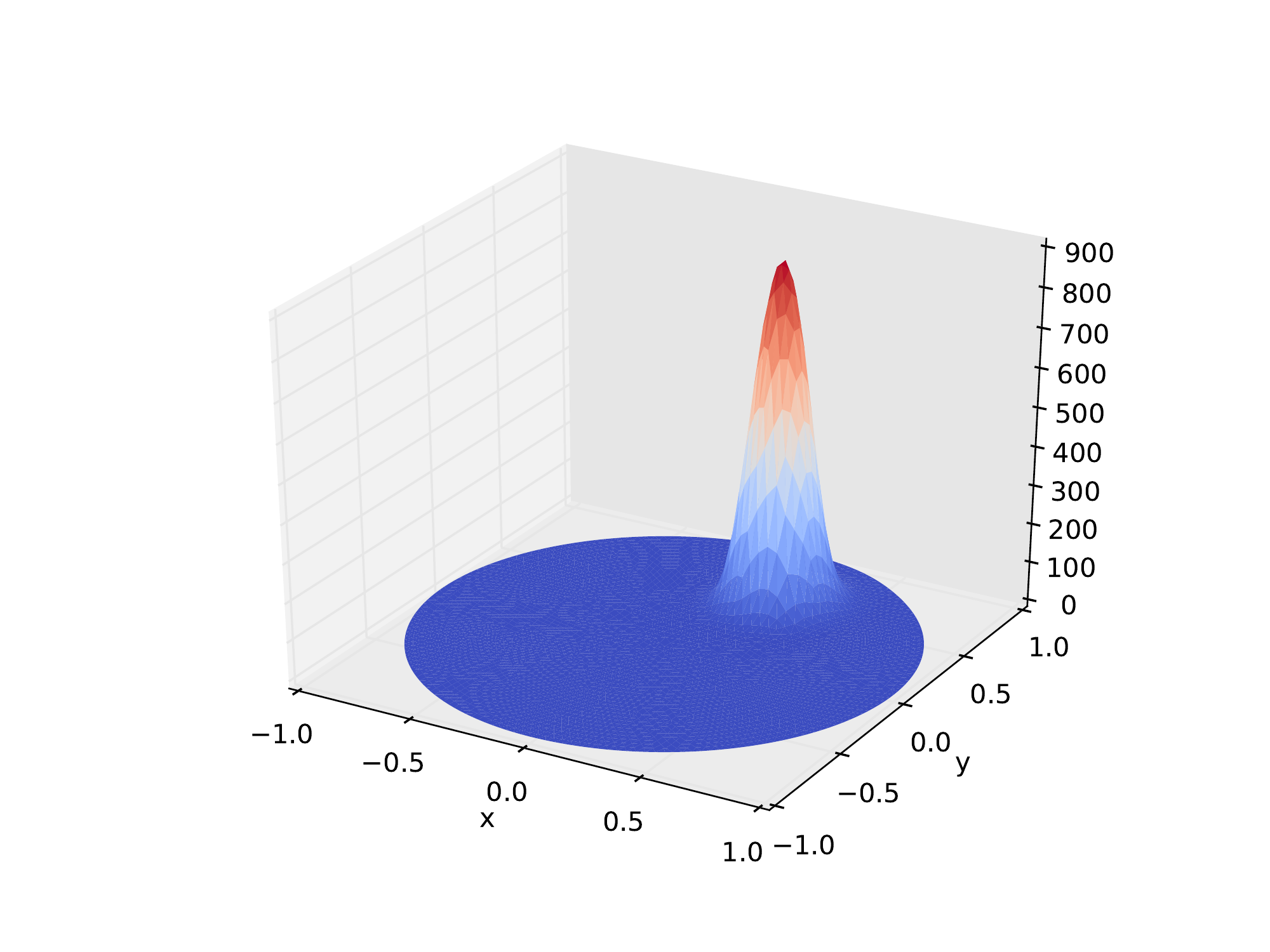}
\includegraphics[width=75mm]{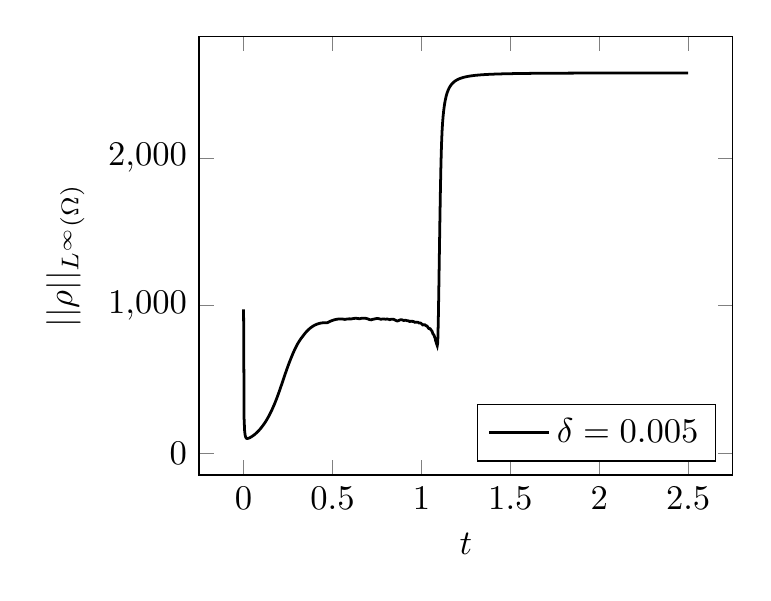}
\caption{Cell density with $\alpha=1$ and $\delta=5\cdot 10^{-3}$ at
times $t=0$ (top left), $t=5\cdot 10^{-3}$ (top right), $t=1$
(bottom left). The $L^\infty$ norm of the density is shown in the bottom right panel.}
\label{fig.exp3}
\end{figure}

{\em Experiment 4: Shape of peaks.}
The previous experiments show that the shape of the peaks depends on the value
of $\delta$. In this experiment, we explore this dependence in more detail.
We claim that the diameter and the height of the bump can be controlled by $\delta$.
We choose $\alpha=1$ and the initial datum $\rho^0=W_{0,0,20\pi}$ with $\theta=1/400$
and $c^0=0$. Furthermore, we prescribe homogeneous Dirichlet boundary
conditions for $c$ to avoid that the aggregated bump of cells moves to the boundary.
Figure \ref{fig.exp4} (top row) shows the stationary cell densities for two values of
$\delta$. As expected, the maximal diameter of the peak (defined at height $10^{-2}$) 
becomes smaller and the maximum of the peak becomes larger
for decreasing values of $\delta$. The level sets show that the solutions
are almost radially symmetric and the level set for $\rho=10^{-2}$ is approximately
a circle. This behavior is quantified in Figure
\ref{fig.exp4} (bottom row). We observe that the radius depends on $\delta$
approximately as $r\sim \delta^{0.43}$ and the height approximately as
$\rho_{\rm max}\sim \delta^{-1.00}$.

We remark that under no-flux boundary conditions for the chemical concentration,
the same behavior of the bumps can be observed for intermediate times. However,
the bump will eventually move to the boundary (as in Figure \ref{fig.exp3}),
since the chemical substance is not absorbed by the boundary as in the
Dirichlet case.

\begin{figure}[ht]
\includegraphics[width=80mm]{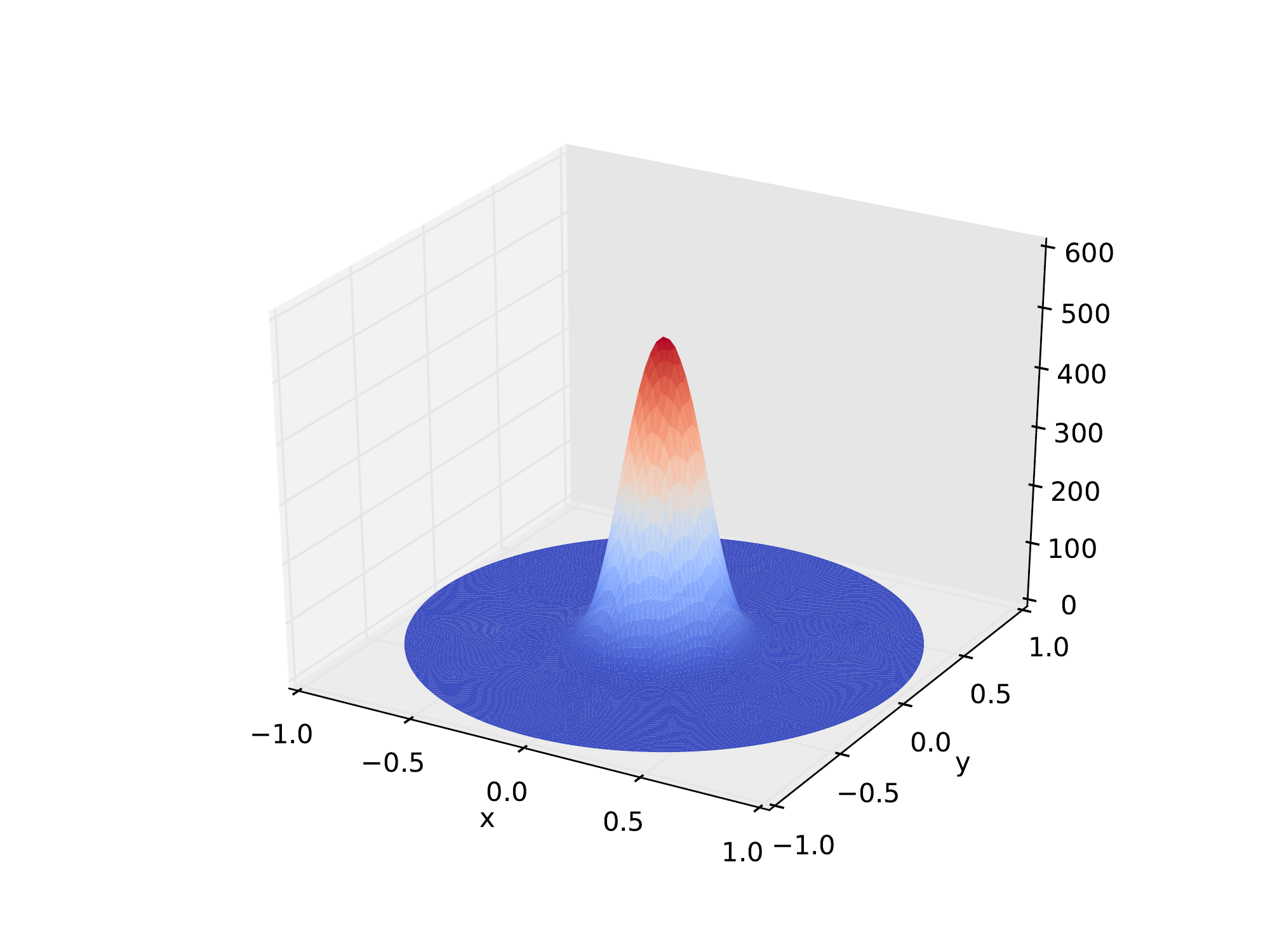}
\includegraphics[width=80mm]{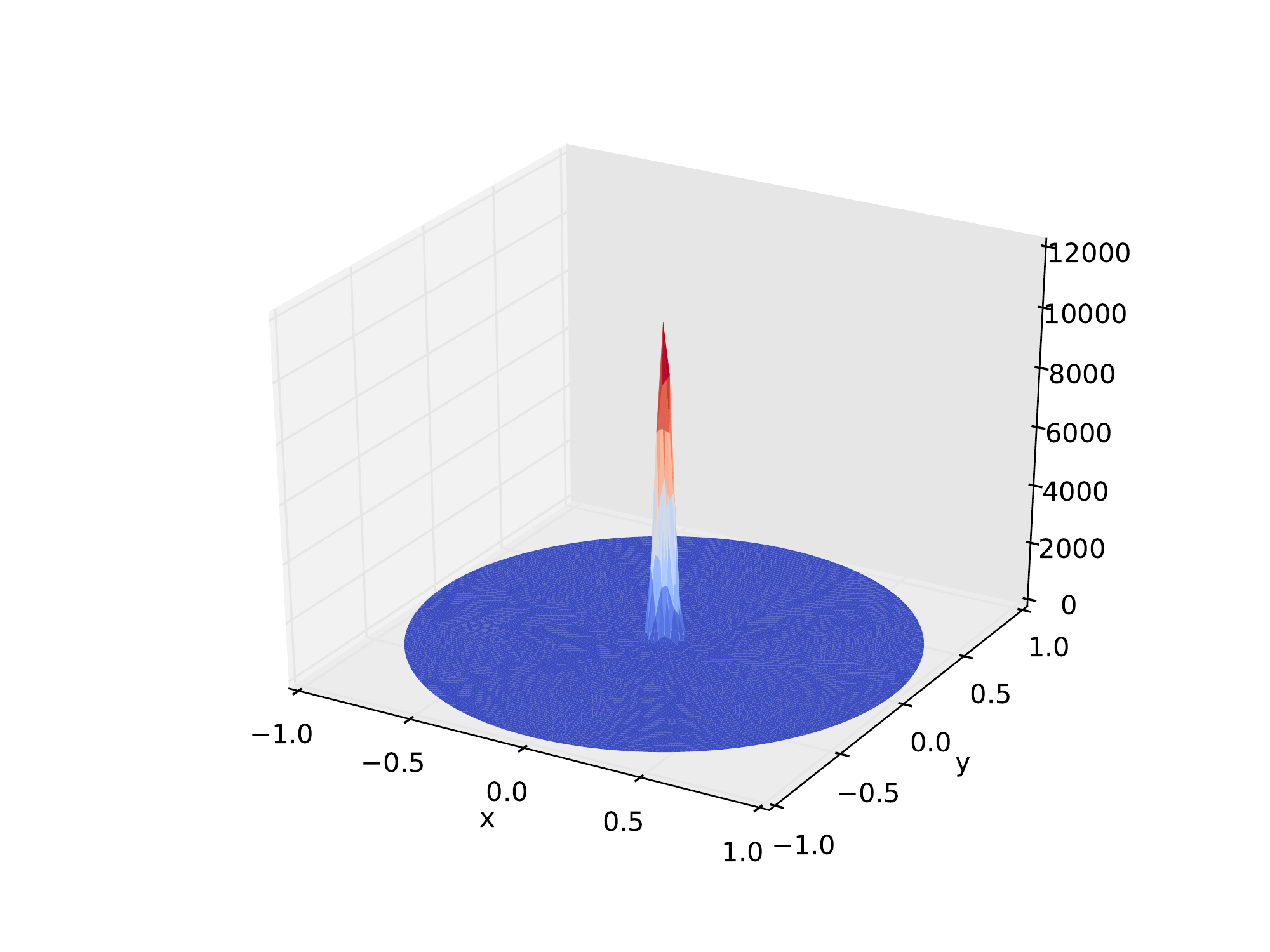}
\includegraphics[width=75mm]{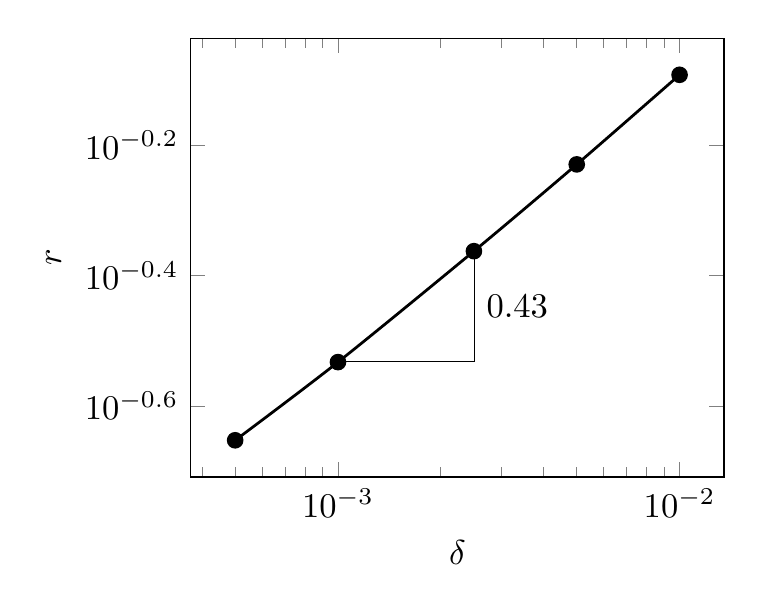}
\includegraphics[width=75mm]{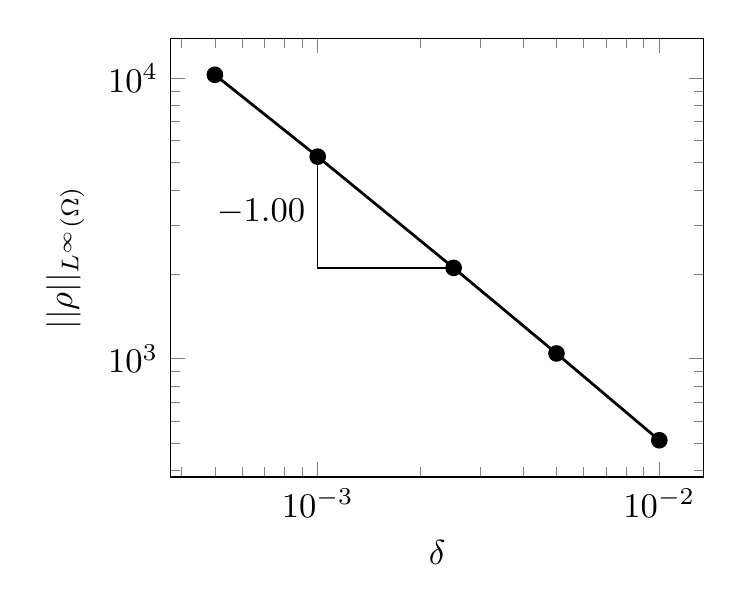}
\caption{Cell density at time $t=5$ (stationary case) with $\alpha=1$ 
and $\delta=10^{-2}$ (top left), $\delta=5\cdot 10^{-4}$ (top right).
Log-log plots of the radius of the density level set 
$\rho=10^{-2}$ versus $\delta$ (bottom left) and
of the maximum of $\rho$ versus $\delta$ (bottom right).}
\label{fig.exp4}
\end{figure}


\begin{appendix}
\section{Some technical tools}\label{sec.tech}

For the convenience of the reader, we collect some technical results.

\begin{lemma}[Inequalities]\label{lem.ineq}
Let $d\le 3$, $\Omega\subset\R^d$ be a bounded domain, and $\pa\Omega\in C^{2,1}$. 
There exists a constant $C>0$ such that for all $u$, $v\in H^1(\Omega)$,
\begin{equation}\label{a.L2}
  \|uv\|_{L^2(\Omega)} \le C\|u\|_{H^1(\Omega)}\|v\|_{H^1(\Omega)},
\end{equation}
for all $u\in H^2(\Omega)$ with $\na u\cdot\nu=0$ on $\pa\Omega$,
\begin{equation}\label{a.H2}
  \|u\|_{H^2(\Omega)}^2 \le C\big(\|\Delta u\|_{L^2(\Omega)}^2
	+\|u\|_{L^2(\Omega)}^2\big),
\end{equation}
and for all $u\in H^3(\Omega)$ with $\na u\cdot\nu=0$ on $\pa\Omega$,
\begin{equation}\label{a.H3}
  \|u\|_{H^3(\Omega)}^2 \le C\big(\|\na\Delta u\|_{L^2(\Omega)}^2
	+\|u\|_{H^2(\Omega)}^2\big).
\end{equation}
\end{lemma}

Inequality \eqref{a.L2} follows after applying the Cauchy--Schwarz inequality
and then the continuous embedding $H^1(\Omega)\hookrightarrow L^4(\Omega)$; 
\eqref{a.H2} is proved
in \cite[Theorem 2.3.3.6]{Gri85}, while \eqref{a.H3} is a consequence of
\cite[Theorem 2.24]{Tro87}.

\begin{lemma}[Nonlinear Gronwall inequality]\label{lem.gronwall}
Let $\delta>0$ and $\Gamma$, $G\in C^0([0,T])$ be nonnegative functions,
possibly depending on $\delta$, satisfying
\begin{align*}
  \Gamma(t) + C_0\int_0^t G(s)ds
	&\le C_1\Gamma(0) + C_2\int_0^t(\Gamma(s)+\Gamma(s)^\alpha)ds \\
	&\phantom{xx}{}
	+ C_3\delta^\beta\int_0^t(\Gamma(s)+\Gamma(s)^\gamma)G(s)ds + C_4\delta^\nu,
\end{align*}
where $\alpha>1$, $\beta\ge 0$, $\gamma>0$, $\nu>0$, and $C_0,\ldots,C_4>0$ 
are constants independent of $\delta$. Furthermore, let 
$\Gamma(0)\le C_5\delta^\nu$ for some $C_5>0$. Then there exists
$\delta_0>0$ such that for all $0<\delta<\delta_0$, $0\le t\le T$, and
$0<\eps<\nu$,
$$
  \Gamma(t) \le C_5\delta^{\nu-\eps}.
$$
\end{lemma}

\begin{proof}
A slightly simpler variant of the lemma was proved in \cite[Lemma 10]{HsWa06}.
Assume, by contradiction, that for all $\delta_0\in(0,1)$, there exist 
$\delta\in(0,\delta_0)$, $t_0\in[0,T]$, and $\eps\in(0,\nu)$ such that
$\Gamma(t_0)> C_5\delta^{\nu-\eps}$. Since $\Gamma(0)\le C_5\delta^{\nu}$ by
assumption and $\Gamma$ is continuous, there exists $t_1\in[0,t_0)$ such that
$\Gamma(t_1)=C_5\delta^{\nu-\eps}$ and $\Gamma(t)\le C_5\delta^{\nu-\eps}$ for
all $t\in[0,t_1]$. This leads for $t\in[0,t_1]$ to
\begin{align*}
  \Gamma(t) + C_0\int_0^t G(s)ds 
	&\le C_1C_5\delta^\nu
	+ C_2\big(1+(C_5\delta^{\nu-\eps})^{\alpha-1}\big)\int_0^t\Gamma(s)ds \\
	&\phantom{xx}{}
	+ C_3\delta^\beta\big(C_5\delta^{\nu-\eps} + (C_5\delta^{\nu-\eps})^\gamma\big)
	\int_0^t G(s)ds + C_4\delta^\nu.
\end{align*}
Since $\nu-\eps>0$, the integral over $G(s)$ on the right-hand side can
be absorbed for sufficiently small $\delta>0$ by the corresponding term on the
left-hand side. This implies that
$$
  \Gamma(t) \le (C_1C_5+C_4)\delta^\nu + 2C_2\int_0^t\Gamma(s)ds, \quad
	0\le t\le t_1.
$$
Then Gronwall's lemma gives, for sufficiently small $\delta_0>0$ and
$0<\delta<\delta_0$,
$$
  \Gamma(t) \le (C_1C_5+C_4)\delta^\nu e^{2C_2 T} \le \frac{C_5}{2}\delta^{\nu-\eps}
	< C_5\delta^{\nu-\eps}, \quad 0\le t\le t_1.
$$
which contradicts $\Gamma(t_1)=C_5\delta^{\nu-\eps}$.
\end{proof}

\end{appendix}


\end{document}